\documentclass[letterpaper,11pt]{article}

\usepackage[utf8]{inputenc}
\usepackage[english]{babel}

\usepackage{latexsym}
\usepackage{amssymb}
\usepackage{amsthm}
\usepackage{mathrsfs}
\usepackage{amstext}
\usepackage{graphicx}  
\usepackage{amsfonts}
\usepackage{wrapfig}
\usepackage{sidecap} 
\usepackage{latexsym}
\usepackage{pdfpages}
\usepackage{epic}
\usepackage{fullpage}
\usepackage{color}
\usepackage{curves}
\usepackage{subfigure}
\usepackage{setspace}
\usepackage{amsmath}
\numberwithin{equation}{section}
\usepackage{anysize}
\usepackage{rotating}
\usepackage{enumerate} 
\usepackage{hyperref}

\usepackage{bbm}
\usepackage{graphicx}
\usepackage{subfigure}

\newtheorem{theorem}{Theorem}[section] 
\newtheorem{definition}[theorem]{Definition}

\newtheorem{proposition}[theorem]{Proposition}
\newtheorem{lemma}[theorem]{Lemma}
\newtheorem{remark}[theorem]{Remark}

\newcommand{\pt}{\partial_t}

\newcommand{\pnu}{\partial_\nu}

\title{An Insensitizing control problem involving tangential gradient terms for a reaction-diffusion equation with dynamic boundary conditions}

\author{
	Mauricio C. Santos\thanks{Department of Mathematics, Federal University of Para\'iba, UFPB, Jo\~{a}o Pessoa, pb cep 58051-900, Brazil, e-mail: mcsantos@mat.ufpb.br}
	\and 
	Nicol\'as Carre\~{n}o\thanks{Departamento de Matem\'atica, Universidad T\'ecnica Federico Santa Mar\'{\i}a, Casilla 110-V, Valpara\'{\i}so, Chile, e-mail: nicolas.carrenog@usm.cl.}
	\and 
	Roberto Morales\thanks{IMUS, Universidad de Sevilla, Campus Reina Mercedes, 41012 Sevilla, Spain, e-mail: rmorales1@us.es }
	}

\begin{document}
\maketitle

\begin{abstract}
In this article, we study the existence of insensitizing controls for a nonlinear reaction-diffusion equation with dynamic boundary conditions. Here, we have a partially unknown data of the system, and the problem consists in finding controls such that a specific functional is insensitive for small perturbations of the initial data. More precisely, the functional considered here depends on the norm of the state in a subset of the bulk together with the norm of the tangential gradient of the state on the boundary. This problem is equivalent to a (relaxed) null controllability problem for an optimality system of cascade type, with a zeroth-order coupling term in the bulk and a second-order coupling term on the boundary. To achieve this result, we linearize the system around the origin and analyze it by the duality approach and we prove a new Carleman estimate for the corresponding adjoint system. Then, a local null controllability result for the nonlinear system is proven by using an inverse function theorem.
\end{abstract}


\section{Introduction}
Let $\Omega\subset \mathbb{R}^d$, $(d\geq 2)$ be a bounded domain with boundary $\Gamma:=\partial \Omega$ of class $C^2$. For $T>0$, we set $Q:=\Omega\times (0,T)$ and $\Sigma :=\Gamma\times (0,T)$. In addition, for each subsets $\omega \subset \Omega$ and $\mathcal{G}\subset \Gamma$, we shall write $Q_\omega :=\omega\times(0,T)$ and $\Sigma_\mathcal{G}:=\mathcal{G}\times (0,T)$.

Given the pair $(R,R_\Gamma)\in L^\infty(Q)\times L^\infty(\Sigma)$, we define the operators $L$ and $L_\Gamma$ by
\begin{align*}
	L(y)=\pt y-\Delta y+Ry\quad \text{and}\quad 
	L_\Gamma(y,y_\Gamma)=\pt y_\Gamma +\pnu y-\Delta_\Gamma y_\Gamma +R_\Gamma y_\Gamma ,
\end{align*}
where $\pnu $ denotes the normal derivative associated to the outward normal $\nu$ of $\Omega$ and $\Delta_\Gamma$ is the Laplace-Beltrami operator acting on $\Gamma$.
In this article, we consider the following reaction-diffusion equation with dynamic boundary conditions
\begin{align}
	\label{eq:intro:01}
	\begin{cases}
		L(y) +|y|^{p-1}y=f+\mathbbm{1}_\omega v&\text{ in }Q,\\
		L_\Gamma(y,y_\Gamma)+|y_\Gamma|^{q-1}y_\Gamma =f_\Gamma&\text{ on }\Sigma,\\
		y\big|_\Gamma=y_\Gamma&\text{ on }\Sigma,\\
		y(\cdot,0)=y^0+\tau_1 \tilde{y}^0&\text{ in }\Omega,\\
		y_\Gamma(\cdot,0)=y_\Gamma^0+\tau_2\tilde{y}_\Gamma^0&\text{ on }\Gamma.  
	\end{cases}
\end{align}

Here, $(y,y_\Gamma)$ is the state of the system, with $y_\Gamma$ being the trace of $y$ on $\Sigma$, $p,q>0$ denotes the order of the reaction, $v\in L^2(Q_\omega)$ is a distributed control function $(\omega\subset \Omega)$, $(f,f_\Gamma)\in L^2(Q)\times L^2(\Sigma)$ are known source terms and $(y^0,y_\Gamma^0)$, $(\hat{y}^0,\hat{y}_\Gamma^0)$ are the initial data and a small perturbation on the system, respectively.

We point out that the control $v\in L^2(Q_\omega)$ of the system \eqref{eq:intro:01} is acting only in the equation satisfied in a small subset of the bulk, i.e., the equation on the boundary is indirectly controlled by the side condition $y\big|_\Gamma=y_\Gamma$ on $\Sigma$. 

Given $\mathcal{O}\subseteq \Omega$ and $\mathcal{G}\subseteq \Gamma$ be two nonempty open subsets, we define the functional $\Phi$ (sometimes refered to the \textit{sentinel}) on the set of solutions of \eqref{eq:intro:01} as follows:
\begin{align}
	\label{def:Phi:insen}
	\Phi(y,y_\Gamma):=\dfrac{1}{2}\iint_{Q_\mathcal{O}} | y(x,t;\tau_1,\tau_2,v)|^2\, dx\,dt + \dfrac{1}{2}\iint_{\Sigma_{\mathcal{G}}} |\nabla_\Gamma y_\Gamma(x,t;\tau_1,\tau_2,v)|^2 \,dS\,dt,
\end{align}
where $(y,y_\Gamma)=(y,y_\Gamma)(x,t;\tau_1,\tau_2,v)$ is the corresponding solution of the system \eqref{eq:intro:01} associated to $\tau_1$, $\tau_2$ and $v$. We say that the couple $(\mathcal{O},\mathcal{G})$ is the \textit{observation domain}.

We emphasize that $\Phi$ is a quadratic functional which depends on the state of the system in the bulk and the tangential derivative of the system on the boundary.

This paper aims to establish the existence of insensitizing controls $v\in L^2(Q_\omega)$ with respect to $\Phi$ in the following sense: 
 
\begin{definition}
	Let $(y^0,y_\Gamma^0)\in L^2(\Omega)\times L^2(\Gamma)$ and $(f,f_\Gamma)\in L^2(Q)\times L^2(\Sigma)$. We say that a control $u\in L^2(Q_\omega)$ insensitizes the functional $\Phi$ if the following condition holds
	\begin{align}
		\label{cond:insen:Phi}
		\dfrac{\partial \Phi }{\partial \tau_1} (\cdot,\cdot,\tau_1,\tau_2,u)\bigg|_{\tau_1=\tau_2=0}=\dfrac{\partial \Phi}{\partial \tau_2} (\cdot,\cdot,\tau_1,\tau_2,u)\bigg|_{\tau_1=\tau_2=0}=0, \quad \forall\, (\hat{y}^0,\hat{y}_\Gamma^0)\in L^2(\Omega)\times L^2(\Gamma),
	\end{align} 
	with $\|(\tilde{y}^0,\tilde{y}_\Gamma^0)\|_{L^2(\Omega)\times L^2(\Gamma)}=1$.
\end{definition}

We point out that the condition \eqref{cond:insen:Phi} denotes that the sentinel $\Phi$ does not detect small perturbations $(\tau \tilde{y}^0,\tau \tilde{y}_\Gamma^0)$ of the initial data $(y^0,y_\Gamma^0)$ on the observation domain. If such control $v\in L^2(Q_\omega)$ exists, we say that $\Phi$ is locally insensitive to small perturbations in the initial data $(y^0,y_\Gamma^0)$. In this case, we say that $v$ \textit{insensitizes} $\Phi$.


The parabolic equations of reactive-diffusive type and dynamic boundary conditions have been intensively studied for the last 20 years \cite{Anguiano2023On} \cite{Gal2012Sharp}, \cite{Anguiano2014Regularity}, \cite{Coclite2009Stability}, \cite{Favini2006Theheat} and the references therein. For a rigorous derivation of these models, we refer to \cite{goldstein2006Derivation} and \cite{gal2015Coleman-Gurtin}. One of the main characteristics of the model \eqref{eq:intro:01} is the presence of the surface diffusion term $\Delta_\Gamma y_\Gamma$, which acts as a damping force in reducing the complexity of the entire system. For a detailed discussion about this subject, see \cite{gal2015Therole}.  

The existence of insensitizing controls in PDEs was originally introduced by J.-L. Lions in \cite{lions1990quelques} and \cite{lions1992sentinelles} and was extended by several authors in different contexts. We mention \cite{bodart1995controls}, where the approximated insensitizing problem was addressed for a nonhomogeneous heat equation with Dirichlet boundary conditions. Besides, in \cite{deteresa2000Insensitizing}, L. de Teresa proved that the existence of insensitizing controls is not guaranteed for every initial data. Moreover, in this article it was proved the existence of insensitizing controls by using Global Carleman estimates introduced by Fursikov and Imanuvilov \cite{fursikov1996imanuvilov}. The advantage of using this approach lies in the fact that Carleman estimates can be used systematically to prove observability inequalities of coupled systems in cascade. These results were generalized later in \cite{bodart2004Existence}, \cite{bodart2004Alocalresult} and \cite{bodart2004Insensitizing} to nonlinearities with suitable superlinear growth and nonlinear terms depending on the state and its gradient. Furthermore, in \cite{deteresa2009Identification} the authors obtained results concerning the class of initial data that can be insensitized. 

We point out that the articles mentioned before consider only quadratic functionals to insensitize and where only the state of the solution is involved. In contrast, in \cite{guerrero2007null} the author consider a different functional which involes the gradient of the state of the linear heat equation. The main tool used is a suitable Carleman estimate of a coupled system with second-order coupling terms. These ideas have been applied to determine the existence of insensitizing controls with more sophisticatef functionals. We mention the work \cite{guerrero2007controllability} for the Stokes system, Navier-Stokes \cite{carreno2014insensitizing}, \cite{gueye2013insensitizing} and \cite{carreno2023Existence}, Boussinesq system \cite{carreno2015Insensitizing} and \cite{carreno2017Insensitizing}, and Ginzburg-Landau equation \cite{santos2019An}.


Regarding the controllability properties of parabolic equations with dynamic boundary conditions, we refer to \cite{maniar2017null} where the authors studied for the first time controllability properties of systems like \eqref{eq:intro:01} and prove a suitable observability inequality by employing Carleman estimates for such heat operator. After that, several articles have been considered generalizations of such equations     
\cite{chorfi2023boundary}, \cite{khoutaibi2020null}, \cite{khoutaibi2022nullsemilinear}, even in the case of fully discrete parabolic equations, see e.g. \cite{lecaros2023discrete}. Concerning insensitizing control problems for parabolic equations with dynamic boundary conditions, in \cite{zhang2019insen} the authors consider the $L^2$ version of \eqref{def:Phi:insen}, i.e., a functional of the form:
\begin{align}
	\label{def:mathcal:J:insen}
	\mathcal{J}(y,y_\Gamma):=\dfrac{1}{2}\iint_{Q_{\mathcal{O}}} | y(x,t;\tau_1,\tau_2,u)|^2 \,dx\,dt + \dfrac{1}{2}\iint_{\Sigma_{\mathcal{G}}} |y_\Gamma(x,t;\tau_1,\tau_2,u)|^2 \,dS\,dt.
\end{align}

In this case, the authors obtained the existence of insensitizing controls for \eqref{def:mathcal:J:insen} by proving a suitable observability inequality for the adjoint of the optimality system. We point out that in this case, such adjoint system is a coupled system in cascade where both coupling terms (in the bulk and on the boundary) are of zeroth order. Therefore, in that case, it is sufficient to apply the Carleman estimate given in \cite{maniar2017null} to each equation and eliminate the local terms where we do not have information. Unfortunately, when we try to insensitize the functional $\Phi$ defined in \eqref{def:Phi:insen} we cannot apply the results of \cite{zhang2019insen} directly since the coupling terms in the adjoint system associated to the optimality system are not the same, see \eqref{Coupled:system:nonlinear:y}-\eqref{Coupled:system:nonlinear:z}. Nevertheless, we will see that a new Carleman estimate can be performed to deal with such new difficulties.

\subsection{Functional setting}

Before going further, we shall need to introduce the notation and functional spaces used in this paper.

The boundary set $\Gamma=\partial \Omega$ can be seen as a $(d-1)$-dimensional compact Riemannian submanifold equipped by the Riemmanian metric induced by the natural embedding $\Gamma \subset \mathbb{R}^d$. 

The tangential gradient $\nabla_\Gamma$ of $y_\Gamma$ at each point $x\in \Gamma$ can be seen as the projection of the standard Euclidean gradient $\nabla y$ on the tangent space of $\Gamma$, where $y_\Gamma$ is the trace of $y$ on $\Gamma$, i.e., 
\begin{align*}
	\nabla_\Gamma y_\Gamma=\nabla y-\nu \pnu y, 
\end{align*}
where $y=y_\Gamma$ on $\Gamma$ and $\pnu y$ is the normal derivative associated to the outward normal $\nu$. In this way, the tangential divergence $\nabla_\Gamma \cdot(\,)$ in $\Gamma$ is defined by
\begin{align*}
	\nabla_\Gamma \cdot(F_\Gamma):H^1(\Gamma) \to \mathbb{R},\quad y_\Gamma \mapsto -\int_\Gamma F_\Gamma \cdot \nabla_\Gamma y_\Gamma \,dS.
\end{align*}

The Laplace-Beltrami operator is given by $\Delta_\Gamma y_\Gamma=\nabla_\Gamma \cdot (\nabla_\Gamma y_\Gamma)$, for all $y_\Gamma \in H^2(\Gamma)$. In particular, we have the following identity, sometimes called the Surface Divergence Theorem:
\begin{align}
	\label{surface:divergence:theorem}
	\int_\Gamma \Delta_\Gamma y_\Gamma z_\Gamma \,dS=-\int_\Gamma \nabla_\Gamma y_\Gamma \cdot \nabla_\Gamma z_\Gamma \,dS,\quad \forall\, y_\Gamma \in H^2(\Gamma),\quad \forall\, z_\Gamma \in H^1(\Gamma).
\end{align}

For $1\leq p\leq +\infty$, we consider the Banach space $\mathbb{L}^p:=L^p(\Omega)\times L^p(\Gamma)$, endowed by the norm 
\begin{align*}
	\|(u,u_\Gamma)\|_{\mathbb{L}^p}:=\left(\|u\|_{L^p(\Omega)}^p + \|u_\Gamma\|_{L^p(\Gamma)}^p\right)^{1/p}\quad \forall \, (u,u_\Gamma)\in \mathbb{L}^p.
\end{align*} 

In particular, $\mathbb{L}^2$ is a Hilbert space equipped with the scalar product 
\begin{align*}
	((u,u_\Gamma),(v,v_\Gamma))_{\mathbb{L}^2}:=\int_\Omega uv\,dx + \int_\Gamma u_\Gamma v_\Gamma \,dS,\quad \forall\, (u,u_\Gamma),(v,v_\Gamma)\in \mathbb{L}^2.
\end{align*}

An immediate consequence of the $L^p$-theory is the following continuous embedding
\begin{align*}
	\mathbb{L}^p \hookrightarrow \mathbb{L}^q\quad \forall\, 1\leq q\leq p\leq \infty.
\end{align*}

In particular, for $0<\theta<1$, let us consider $1\leq q_1<r_1<p_1\leq \infty$ satisfying
\begin{align*}
	\frac{1}{r_1}=\frac{\theta}{p_1}+\dfrac{1-\theta}{q_1}.
\end{align*}

Moreover, consider $1\leq p_2<r_2<q_2\leq \infty$ such that 
\begin{align*}
	\frac{1}{r_2}=\frac{\theta}{p_2}+\dfrac{1-\theta}{q_2}.
\end{align*}

Then, we have the following continuous embedding
\begin{align*}
	L^{r_1}(0,T;\mathbb{L}^{r_2})\hookrightarrow L^{p_1}(0,T;\mathbb{L}^{p_2}) \cap L^{q_1}(0,T; \mathbb{L}^{q_2}).
\end{align*}

Moreover, there exists a constant $C>0$ such that following inequality holds 
\begin{align*}
	\|(u,u_\Gamma)\|_{L^{r_1}(0,T;\mathbb{L}^{r_2})}\leq C \|(u,u_\Gamma)\|_{L^{p_1}(0,T;\mathbb{L}^{p_2})}^{\theta} \cdot \|(u,u_\Gamma)\|_{L^{q_1}(0,T;\mathbb{L}^{q_2})}^{1-\theta}. 
\end{align*}
for all $(u,u_\Gamma)\in L^{p_1}(0,T;\mathbb{L}^{p_2})\cap L^{p_2}(0,T;\mathbb{L}^{q_2})$.

For $k\in \mathbb{N}$, we also introduce the space 
\begin{align*}
	\mathbb{H}^k:=\{(y,y_\Gamma)\in H^k(\Omega)\times H^k(\Gamma)\,:\, y\big|_\Gamma =y_\Gamma\},
\end{align*}
where $H^k(\Omega)$ and $H^k(\Gamma)$ are the usual Sobolev spaces in $\Omega$ and $\Gamma$, respectively.

Since we are considering parabolic problems, for $T>0$ we introduce the Hilbert space 
\begin{align*}
	\mathbb{E}:=H^1(0,T;\mathbb{L}^2)\cap L^2(0,T;\mathbb{H}^2),
\end{align*}
endowing with his natural inner product. It is well-known that $C^0([0,T];\mathbb{H}^1) \hookrightarrow \mathbb{E}$ continuously. Moreover, $\mathbb{E}$ embeds compactly in $L^2(0,T;\mathbb{L}^2)$ (see for instance \cite[Proposition 2.2]{maniar2017null}).

For our purposes, we also introduce the Banach space 
\begin{align*}
	\mathbb{D}:=C^0([0,T];\mathbb{H}^1)\cap L^2(0,T;\mathbb{H}^2),
\end{align*}
endowed by its natural norm. Besides, we can also deduce interpolation results for the spaces $L^r(0,T;\mathbb{L}^s)$ and $L^r(0,T;\mathbb{H}^m)$ combining the results of Adams \cite[Theorem 4.12]{Adams2003Sobolev} and Lee \cite[Theorem 2.2]{Lee1987Yamabe}. 

On the other hand, in this article we are interested in the following situations:
\begin{itemize}
	\item If $d=2$, then $L^\infty(0,T;\mathbb{L}^r) \cap L^2(0,T;\mathbb{L}^\infty) \hookrightarrow \mathbb{D}$ continuously for all $1\leq r<+\infty$.
	\item If $d=3$, then $L^\infty(0,T;\mathbb{L}^6) \cap L^2(0,T;\mathbb{L}^\infty) \hookrightarrow \mathbb{D}$ continuously.
\end{itemize}  

\subsection{Main results}

The following assumptions will be needed throughout the paper:
\begin{itemize}
	\item[{\bf (H1)}] Suppose that $\omega \Subset \Omega$ and $\mathcal{O}\subseteq \Omega$ are open sets satisfying $\omega\cap \mathcal{O}\neq \emptyset$ and $\mathcal{G}\subseteq \Gamma$ is an arbitrary open set.
	\item[{\bf (H2)}] The coefficients $p,q>0$ in equation \eqref{eq:intro:01} satisfies the following conditions:
	\begin{itemize}
		\item[(i)] If $d=2$, then $\frac{5}{2}< p,q <+\infty$.
		\item[(ii)] If $d=3$, then $\frac{5}{2}< p,q\leq 4$.
	\end{itemize}
\end{itemize}

Then, the main result of this paper can be stated as follows:

\begin{theorem}
\label{Thm:main:result}
	Assume that $d=2$ or $d=2$ and suppose that conditions {\bf (H1)} and {\bf (H2)} hold. Assume that $(y^0,y_\Gamma^0)\equiv (0,0)$. Then, there exist constants $M_1>0$ and $\delta>0$ such that for any $(f,f_\Gamma)\in L^2(0,T;\mathbb{L}^2)$ satisfying
	\begin{align}
	\label{cond:f:fgamma}
		\|e^{\frac{M_1}{t}} (f,f_\Gamma)\|_{L^2(0,T;\mathbb{L}^2)}\leq \delta,
	\end{align}
	there is a control $v\in L^2(Q_\omega)$ that insensitizes the functional $\Phi$ defined by \eqref{def:Phi:insen}.
\end{theorem}
%

Now, let us introduce the formal adjoint operators of $L$ and $L_\Gamma$, respectively:
\begin{align*}
	L^*z:=-\pt z-\Delta z+Rz,\quad L_\Gamma^*(z,z_\Gamma):= -\pt z_\Gamma +\pnu z-\Delta_\Gamma z_\Gamma + R_\Gamma z_\Gamma. 
\end{align*}

It is well-known that the existence of insensitizing controls can be analyzed by a (relaxed) controllability problem for a suitable coupled system with a second-order coupling term. More precisely, our problem is equivalent to find a $v\in L^2(Q_\omega)$ such that the (optimality) system
\begin{align}
	\label{Coupled:system:nonlinear:y}
	\begin{cases}
		L(y) + |y|^{p-1}y=f+\mathbbm{1}_\omega v &\text{ in }Q,\\
		L_\Gamma(y,y_\Gamma)+|y_\Gamma|^{q-1}y_\Gamma =f_\Gamma &\text{ on }\Sigma,\\
		y\big|_\Gamma=y_\Gamma &\text{ on }\Sigma,\\
		(y(\cdot,0),y_\Gamma(\cdot,0))=(y^0,y_\Gamma^0) &\text{ in }\Omega\times \Gamma,
	\end{cases}
\end{align}
\begin{align}
	\label{Coupled:system:nonlinear:z}
	\begin{cases}
		L^*(z) + p|y|^{p-2}yz= \mathbbm{1}_\mathcal{O}y &\text{ in }Q,\\
		L_\Gamma^*(z,z_\Gamma) +q|y_\Gamma|^{q-2}y_\Gamma z_\Gamma=\nabla_\Gamma\cdot (\mathbbm{1}_\mathcal{G}\nabla_\Gamma y_\Gamma) &\text{ on }\Sigma,\\
		z\big|_\Gamma=z_\Gamma &\text{ on }\Sigma,\\
		(z(\cdot,T),z_\Gamma(\cdot,T))=(0,0)&\text{ in }\Omega\times \Gamma,
	\end{cases}
\end{align}
drives the state $(z,z_\Gamma)$ to the rest at time $t=0$.

Then, following the arguments presented in \cite[Proposition 1]{bodart1995controls}, we have:
\begin{theorem}
\label{thm:nonlinear:controllability:coupled:system}
Assume that $d=2$ or $d=3$ and suppose that {\bf (H1)} and {\bf (H2)} hold and $(y^0,y_\Gamma^0)=(0,0)$. Moreover, assume that there are constants $M_1,\delta>0$ such that $(f,f_\Gamma)$ satisfies the condition \eqref{cond:f:fgamma}. 
	If there exists a control $v\in L^2(Q_\omega)$ such that the corresponding solution $(y,y_\Gamma,z,z_\Gamma)$ of the coupled system \eqref{Coupled:system:nonlinear:y}-\eqref{Coupled:system:nonlinear:z} verifies 
	\begin{align*}
	z(\cdot,0)=0\text{ in }\Omega,\quad z_\Gamma(\cdot,0)=0\text{ on }\Gamma,
\end{align*}
then $v\in L^2 (Q_\omega)$ insensitizes the functional $\Phi$ in the sense of definition \eqref{def:Phi:insen}. The converse is also true.
\end{theorem}

The proof of Theorem \ref{thm:nonlinear:controllability:coupled:system} relies on a local inverse mapping theorem in Banach spaces. Therefore, we shall consider a linear system of the form 
\begin{align}
	\label{coupled:linear:system:y}
	\begin{cases}
		L(y)=f^0+\mathbbm{1}_\omega v &\text{ in }Q,\\
		L(y,y_\Gamma) =f_\Gamma^0 &\text{ in }\Sigma,\\
		y\big|_\Gamma=y_\Gamma &\text{ on }\Sigma,\\
		(y(\cdot,0),y_\Gamma(\cdot,0))=(y^0,y_\Gamma^0) &\text{ in }\Omega\times \Gamma,
	\end{cases}
\end{align}
\begin{align}
	\label{coupled:linear:system:z}
	\begin{cases}
		L^*(z)=f^1 + \mathbbm{1}_{\mathcal{O}}y &\text{ in }Q,\\
		L_\Gamma^*(z,z_\Gamma)=f_\Gamma^1 + \nabla_\Gamma\cdot (\mathbbm{1}_\mathcal{G} \nabla_\Gamma y_\Gamma)&\text{ on }\Sigma,\\
		z\big|_\Gamma=z_\Gamma &\text{ on }\Sigma,\\
		(z(\cdot,T),z_\Gamma(\cdot,T))=(0,0) &\text{ in }\Omega\times \Gamma.
	\end{cases}
\end{align}

Then, we prove that for any $(f^0,f_\Gamma^0),(f^1,f_\Gamma^1)$ in a suitable weighted space, the couple system \eqref{coupled:linear:system:y}-\eqref{coupled:linear:system:z} is null controllable at any time $T>0$. To do that, we will prove a new Carleman estimate for an associated state of \eqref{coupled:linear:system:y}-\eqref{coupled:linear:system:z}. Finally, we combine this result with an inverse mapping theorem to conclude the local null controllability of \eqref{Coupled:system:nonlinear:y}-\eqref{Coupled:system:nonlinear:z}.

The rest of the paper is organized as follows. In Section \ref{section:existence:results} we study existence and uniqueness results for systems in \eqref{coupled:linear:system:y}-\eqref{coupled:linear:system:z} and \eqref{Coupled:system:nonlinear:y}-\eqref{Coupled:system:nonlinear:z}. In Section \ref{section:Carleman:inequalities} we prove a new Carleman estimate for a parabolic system with dynamic boundary conditions. In Section \ref{section:Null:controllability}, we prove a null controllability result for the system \eqref{coupled:linear:system:y}-\eqref{coupled:linear:system:z} by duality arguments. In Section \ref{section:proof:main:theorem}, we prove the main result of the paper, i.e., Theorem \ref{Thm:main:result}. Finally, in Section \ref{section:further:comments}, additional comments related to our results are given.


\section{Existence and uniqueness results}
\label{section:existence:results}
In this section, we provide some elementary existence and uniqueness results concerning systems of parabolic equations with dynamic boundary conditions.

\subsection{Elementary results}
We start recalling existence and uniqueness result on the linear heat equation with dynamic boundary conditions. Roughly speaking, it relies on the analyticity of the Wentzell Laplacian operator and the standard semigroup theory and it can be found in \cite{maniar2017null}. 

Consider the linear heat equation with dynamic boundary conditions: 
\begin{align}
	\label{eq:existence:01}
	\begin{cases}
		L(y)=h&\text{ in }Q,\\
		L_\Gamma(y,y_\Gamma)=h_\Gamma &\text{ on }\Sigma,\\
		y\big|_\Gamma=y_\Gamma &\text{ on }\Sigma,\\
		(y(\cdot,0),y_\Gamma(\cdot,0))=(y^0,y_\Gamma^0) &\text{ in }\Omega\times \Gamma.
	\end{cases}
\end{align}

We recall the space $\mathbb{E}:=H^1(0,T;\mathbb{L}^2)\cap L^2(0,T;\mathbb{H}^2)$.

\begin{definition}
	Suppose that $(h,h_\Gamma)\in L^2(0,T;\mathbb{L}^2)$ and $(y^0,y_\Gamma^0)\in \mathbb{L}^2$.
	\begin{itemize}
		\item A strong solution of \eqref{eq:existence:01} is a pair $(y,y_\Gamma)\in \mathbb{E}$ satisfying \eqref{eq:existence:01} in $L^2(0,T;\mathbb{L}^2)$.
		\item A distributional solution of \eqref{eq:existence:01} is a pair $(y,y_\Gamma)\in L^2(0,T;\mathbb{L}^2)$ such that for any $(\varphi,\varphi_\Gamma)\in \mathbb{E}$ with $(\varphi(\cdot,T),\varphi_\Gamma(\cdot,T))=(0,0)$, the following identity holds:
		\begin{align*}
			\begin{split}
				&\iint_{Q} yL^*(\varphi)\,dx\,dt + \iint_{\Sigma} y_\Gamma L^*_{\Gamma}(y,y_\Gamma) -\int_\Omega y^0\varphi(\cdot,0)\,dx
				-\int_\Gamma y_\Gamma^0\varphi_\Gamma(\cdot,0)\,dS \\
				&=\iint_{Q} h\varphi \,dx\,dt + \iint_{\Sigma} h_\Gamma \varphi_\Gamma \,dS\,dt
			\end{split}
		\end{align*} 
	\end{itemize}
\end{definition}

The next result establishes the existence and uniqueness of solutions $(y,y_\Gamma)$ of \eqref{eq:existence:01}.
\begin{proposition}[see \cite{maniar2017null}]
	\label{prop:existence:heat:eq}
	Suppose that $(h,h_\Gamma)\in L^2(0,T;\mathbb{L}^2)$.
	\begin{itemize}
		\item If $(y^0,y_\Gamma^0)\in \mathbb{L}^2$, then there exists a unique distributional solution $(y,y_\Gamma)\in C^0([0,T];\mathbb{L}^2)\cap L^2(0,T;\mathbb{H}^1)$ of \eqref{eq:existence:01}. Moreover, there exists a constant $C>0$ such that 
		\begin{align*}
			\begin{split} 
			&\|(y,y_\Gamma)\|_{C^0([0,T];\mathbb{L}^2)} + \|(y,y_\Gamma)\|_{L^2(0,T;\mathbb{H}^1)}\\
			\leq & C\|(y^0,y_\Gamma^0)\|_{\mathbb{L}^2} + C\|(h,h_\Gamma)\|_{L^2(0,T;\mathbb{L}^2)}.
			\end{split}
		\end{align*}
		\item If $(y^0,y_\Gamma^0)\in \mathbb{H}^1$, then there exists a unique strong solution $(y,y_\Gamma)\in \mathbb{E}$. Moreover, there exists a constant $C>0$ such that 
		\begin{align*}
			\|(y,y_\Gamma)\|_{\mathbb{E}}\leq C\|(y^0,y_\Gamma^0)\|_{\mathbb{H}^1} + C\|(h,h_\Gamma)\|_{L^2(0,T;\mathbb{L}^2)}.
		\end{align*}
		\item If $(y,y_\Gamma)\in C^0([0,T];\mathbb{L}^2)$ is a distributional solution and $(y^0,y_\Gamma^0)\in \mathbb{H}^1$, then it is a strong solution and we have $(y,y_\Gamma)\in \mathbb{E}$.
	\end{itemize}
\end{proposition}

\subsection{The linear coupled system}

For each non-empty open sets $\mathcal{O}\subseteq \Omega$ and $\mathcal{G}\subseteq \Gamma$, consider the following linear coupled system in cascade:
\begin{align}
	\label{coupled:system:existence:linear:y}
	\begin{cases}
		L(y)=h^0&\text{ in }Q,\\
		L_\Gamma(y,y_\Gamma)=h_\Gamma^0&\text{ on }\Sigma,\\
		y\big|_\Gamma=y_\Gamma&\text{ on }\Sigma,\\
		(y(\cdot,0),y_\Gamma(\cdot,0))=(y^0,y_\Gamma^0)&\text{ in }\Omega\times \Gamma,
	\end{cases}
\end{align}

\begin{align}
	\label{coupled:system:existence:linear:z}
	\begin{cases}
		L^*(z)=h^1+\mathbbm{1}_{\mathcal{O}}y&\text{ in }Q,\\
		L_\Gamma^*(z,z_\Gamma)=h_\Gamma^0+\nabla_\Gamma \cdot (\mathbbm{1}_{\mathcal{G}} \nabla_\Gamma y_\Gamma)&\text{ on }\Sigma,\\
		z\big|_\Gamma=z_\Gamma&\text{ on }\Sigma,\\
		(z(\cdot,T),z_\Gamma(\cdot,T))=(z^0,z_\Gamma^0)&\text{ in }\Omega\times \Gamma.
	\end{cases}
\end{align}

\begin{definition}
	Suppose that $(h^0,h_\Gamma^0),(h^1,h_\Gamma^1)\in L^2(0,T;\mathbb{L}^2)$ and $(y^0,y_\Gamma^0),(z^0,z_\Gamma^0)\in \mathbb{L}^2$, we say that
	\begin{itemize}
		\item the couple $(y,y_\Gamma,z,z_\Gamma)\in \mathbb{E}$ is a strong solution if \eqref{coupled:system:existence:linear:y}-\eqref{coupled:system:existence:linear:z} is satisfied in $[L^2(0,T;\mathbb{L}^2)]^2$.
		\item $(y,y_\Gamma,z,z_\Gamma)\in [L^2(0,T;\mathbb{L}^2)]^2$ is a distributional solution of \eqref{coupled:system:existence:linear:y}-\eqref{coupled:system:existence:linear:z} if the following identity holds
		\begin{align*}
			\begin{split}
				&\iint_{Q}(yL^*\varphi + zL(\psi))\,dx\,dt + \iint_{\Sigma} (y_\Gamma L_\Gamma^*(\varphi,\varphi_\Gamma) +z_\Gamma L_\Gamma(\psi,\psi_\Gamma))\,dS\,dt\\
				&-\int_\Omega (y^0\varphi(\cdot,0)-z^0\psi(\cdot,0))\,dx
				-\int_{\Gamma} (y_\Gamma^0\varphi_\Gamma(\cdot,0) - z_\Gamma^0\psi_\Gamma(\cdot,T))\,dS\\
				=&\iint_Q (h^0\varphi + h^1\psi)\,dx\,dt + \iint_{\Sigma} (h_\Gamma^0 \varphi_\Gamma + h_\Gamma^1 \psi_\Gamma)\,dS\,dt\\
				&+\iint_{Q} y\mathbbm{1}_{\mathcal{O}} \psi \,dx\,dt + \iint_{\Sigma} y_\Gamma \nabla_\Gamma \cdot (\mathbbm{1}_{\mathcal{G}} \nabla_\Gamma \psi_\Gamma)\,dS\,dt
			\end{split}
		\end{align*} 
		holds for all $(\varphi,\varphi_\Gamma,\psi,\psi_\Gamma)\in \mathbb{E}^2$ with $(\varphi(\cdot,T),\varphi_\Gamma(\cdot,T))=(0,0)$ and $(\psi(\cdot,0),\psi_\Gamma(\cdot,0))=(0,0)$.
	\end{itemize}	
\end{definition}

Then, we have the following result:

\begin{proposition}
	\label{prop:existence:coupled:system}
	Suppose that $(h^0,h_\Gamma^0),(h^1,h_\Gamma^1)\in L^2(0,T;\mathbb{L}^2)$. In addition, let $\Omega\subseteq \Omega$ and $\mathcal{G}\subseteq \Gamma$ be two non-empty open sets.
	\begin{itemize}
		\item If $(y^0,y_\Gamma^0),(z^0,z_\Gamma^0)\in \mathbb{L}^2$, then there exists a unique distributional solution $(y,y_\Gamma,z,z_\Gamma)\in [C^0([0,T];\mathbb{L}^2)\cap L^2(0,T;\mathbb{H}^1)]^2$ of \eqref{coupled:system:existence:linear:y}-\eqref{coupled:system:existence:linear:z}. Moreover, there exists a constant $C>0$ such that 
		\begin{align*}
			&\|(y,y_\Gamma)\|_{C([0,T];\mathbb{L}^2)} + \|(z,z_\Gamma)\|_{C([0,T];\mathbb{L}^2)}\\
			\leq  &C\|(h^0,h_\Gamma^0)\|_{L^2(0,T;\mathbb{L}^2)} + C\|(h^1,h_\Gamma^1)\|_{L^2(0,T;\mathbb{L}^2)}+C\|(y^0,y_\Gamma^0)\|_{\mathbb{L}^2} + C\|(z^0,z_\Gamma^0)\|_{\mathbb{L}^2}.
		\end{align*}
		\item If $(y^0,y_\Gamma^0),(z^0,z_\Gamma^0)\in \mathbb{H}^1$, then there exists a unique weak solution $(y,y_\Gamma,z,z_\Gamma)\in \mathbb{E}^2$. Moreover, there exists a constant $C>0$ such that 
		\begin{align}
			\label{energy:estimate:H1}
			\begin{split} 
			&\|(y,y_\Gamma)\|_{\mathbb{E}} + \|(z,z_\Gamma)\|_{\mathbb{E}}\\
			\leq  &C\|(h^0,h_\Gamma^0)\|_{L^2(0,T;\mathbb{L}^2)} + C\|(h^1,h_\Gamma^1)\|_{L^2(0,T;\mathbb{L}^2)}+C\|(y^0,y_\Gamma^0)\|_{\mathbb{H}^1} + C\|(z^0,z_\Gamma^0)\|_{\mathbb{H}^1}.
			\end{split} 
		\end{align}
		\item If $(y,y_\Gamma,z,z_\Gamma)$ is a distributional solution and $(y^0,y_\Gamma^0),(z^0,z_\Gamma^0)\in \mathbb{H}^1$, then it is a strong solution with $(y,y_\Gamma,z,z_\Gamma)\in \mathbb{E}^2$. 
	\end{itemize}
\end{proposition}

The proof of Proposition \ref{prop:existence:coupled:system} follows easily from Proposition \ref{prop:existence:heat:eq}, the cascade coupling structure of \eqref{coupled:system:existence:linear:y}-\eqref{coupled:system:existence:linear:z} and the use of the Surface Divergence Theorem \eqref{surface:divergence:theorem}. The details are left to the reader.

\subsection{The nonlinear coupled system}

The remainder of this section will be devoted to establish existence and uniqueness results of the nonlinear heat system of reactive-diffusive type with dynamic boundary conditions. 

Now, consider a nonlinear system of the form
\begin{align}
	\label{system:nonlinear:01}
	\begin{cases} 
	L(y)+|y|^{\sigma-1}y=h^0&\text{ in }Q,\\
	L_\Gamma(y,y_\Gamma)+|y|^{\kappa-1}y_\Gamma=h_\Gamma^0&\text{ on }\Sigma,\\
	y\big|_\Gamma=y_\Gamma&\text{ on }\Sigma,\\
	(y(\cdot,0),y_\Gamma(\cdot,0)) =(y^0,y_\Gamma^0)&\text{ in }\Omega\times \Gamma,
	\end{cases}
\end{align}
\begin{align}
	\label{system:nonlinear:02}
	\begin{cases}
		L^*(z)+|y|^{\sigma-2}yz=h^1+\mathbbm{1}_{\mathcal{O}}y&\text{ in }Q,\\
		L_\Gamma^*(z,z_\Gamma)+|y_\Gamma|^{\kappa-2}y_\Gamma z_\Gamma=h_\Gamma^1 + \nabla_\Gamma\cdot \left(\mathbbm{1}_{\mathcal{G}}\nabla_\Gamma y_\Gamma  \right)&\text{ on }\Sigma,\\
		z\big|_\Gamma=z_\Gamma&\text{ on }\Sigma,\\
		(z(\cdot,T),z_\Gamma(\cdot,T))=(z^0,z_\Gamma^0)&\text{ in }\Omega\times \Gamma.
	\end{cases}
\end{align}

For our purposes, we consider the following conditions:
\begin{itemize}
	\item[{\bf (A1)}] The sets $\mathcal{O}\subseteq \Omega$ and $\mathcal{G}\subseteq \Gamma$ are two non-empty open subsets.
	\item[{\bf (A2)}] The parameters $\sigma$ and $\kappa$ satisfies the following conditions:
	\begin{itemize}
		\item[(i)] If $d=2$, then $ \frac{9}{4}\leq \sigma,\kappa< \infty$.
		\item[(ii)] If $d=3$, then $\frac{9}{4}\leq \sigma,\kappa\leq 4$. 
	\end{itemize} 
\end{itemize}

We recall that $\mathbb{D}:=C^0([0,T];\mathbb{H}^1)\cap L^2(0,T;\mathbb{H}^2)$. Then, we have the following result:

\begin{proposition}
	\label{prop:existence:nonlinear}
	Assume that $d=2$ or $d=3$.  and suppose that {\bf (A1)} and {\bf (A2)} hold. Then, there exists $r>0$ and $\epsilon>0$ such that for every $(h^0,h_\Gamma^1),(h^1,h_\Gamma^1)\in L^2(0,T;\mathbb{L}^2)$ and $(y^0,y_\Gamma^0),(z^0,z_\Gamma^0)\in \mathbb{H}^1$ satisfying
	\begin{align}
		\label{assumption:smallness:data}
		\|(h^0,h_\Gamma^0)\|_{L^2(0,T;\mathbb{L}^2)}+\|(h^1,h_\Gamma^1)\|_{L^2(0,T;\mathbb{L}^2)} + \|(y^0,y_\Gamma^0)\|_{\mathbb{H}^1}+\|(z^0,z_\Gamma^0)\|_{\mathbb{L}^2}\leq \epsilon,
	\end{align}
	there exists a unique solution $(y,y_\Gamma,z,z_\Gamma)\in \mathbb{D}^2$ of \eqref{system:nonlinear:01}-\eqref{system:nonlinear:02} with $\|(y,y_\Gamma,z,z_\Gamma)\|_{\mathbb{D}^2}\leq r$.
\end{proposition}

\begin{proof}
	We will argue using fixed-point arguments. Given $(u,u_\Gamma,v,v_\Gamma)\in \mathbb{D}^2$, we consider the linear system:
	\begin{align}
		\label{system:nonlinear:03}
		\begin{cases}
			L(y)=h^0-|u|^{\sigma-1}u &\text{ in }Q,\\
			L_\Gamma(y,y_\Gamma)=h_\Gamma^0-|u_\Gamma|^{\kappa-1}u_\Gamma &\text{ on }\Sigma,\\
			y\big|_\Gamma=y_\Gamma&\text{ on }\Sigma,\\
			(y(\cdot,0),y_\Gamma(\cdot,0))=(y^0,y_\Gamma^0)&\text{ in }\Omega\times \Gamma.
		\end{cases}
	\end{align}
	\begin{align}
		\label{system:nonlinear:04}
		\begin{cases}
			L^*(z)=h^1+\mathbbm{1}_{\mathcal{O}}y-|u|^{\sigma-2}uv&\text{ in }Q,\\
			L_\Gamma^*(z,z_\Gamma)=h_\Gamma^1+\nabla_\Gamma\cdot(\mathbbm{1}_{\mathcal{G}} \nabla_\Gamma y_\Gamma)-|u_\Gamma|^{\kappa-2}u_\Gamma v_\Gamma&\text{ on }\Sigma,\\
			z\big|_\Gamma=z_\Gamma&\text{ on }\Sigma,\\
			(z(\cdot,T),z_\Gamma(\cdot,T))=(z^0,z_\Gamma^0)&\text{ in }\Omega\times \Gamma.
		\end{cases}
	\end{align} 
	
	In addition, let us define the mapping $\mathcal{F}:\mathbb{D}^2\to \mathbb{D}^2$ by
	\begin{align*}
		\mathcal{F}(u,u_\Gamma,v,v_\Gamma)=(y,y_\Gamma,z,z_\Gamma),\quad \forall\, (u,u_\Gamma,v,v_\Gamma)\in \mathbb{D}^2,
	\end{align*}
	where $(y,y_\Gamma,z,z_\Gamma)$ is the unique solution of \eqref{system:nonlinear:03}-\eqref{system:nonlinear:04} associated to $(u,u_\Gamma,v,v_\Gamma)$. 
	Now, it is clear that $(y,y_\Gamma,z,z_\Gamma)\in \mathbb{D}^2$ is a solution of \eqref{system:nonlinear:03}-\eqref{system:nonlinear:04} if and only if it is a fixed point of $\mathcal{F}$.
	
	Our next task is to show that there exists $0<r<1$ such that the restriction of $\mathcal{F}$ to the closed ball $\mathbb{B}_r:=\{(y,y_\Gamma,z,z_\Gamma)\in \mathbb{D}^2\,:\, \|(y,y_\Gamma,z,z_\Gamma)\|_{\mathbb{D}^2}\leq r \}$ is a contraction from $\mathbb{B}_r$ into $\mathbb{B}_r$.
	
	In the following computations, we just consider the case $d=3$. The case $d=2$ is analogous and is left to the reader. Let $(u,u_\Gamma,v,v_\Gamma)\in \mathbb{D}^2$. Then, by Proposition \ref{prop:existence:coupled:system}, there is a unique solution $(y,y_\Gamma,z,z_\Gamma)$ of \eqref{system:nonlinear:03}-\eqref{system:nonlinear:04}. Moreover, from the energy estimate \eqref{energy:estimate:H1}, we have 
	\begin{align}
		\label{estimate:mathcal:F}
		\begin{split} 
		&\|\mathcal{F}(u,u_\Gamma,v,v_\Gamma)\|_{\mathbb{D}^2}\\
		\leq & C(\|(h^0,h_\Gamma^0)\|_{L^2(0,T;\mathbb{L}^2)} + \|(h^1,h_\Gamma^1)\|_{L^2(0,T;\mathbb{L}^2)} + \|(y^0,y_\Gamma^0)\|_{\mathbb{H}^1} + \|(z^0,z_\Gamma^0)\|_{\mathbb{H}^1}\\
		&+\|(|u|^\sigma,|u_\Gamma|^\kappa)\|_{L^2(0,T;\mathbb{L}^2)} + \|(|u|^{\sigma-1}v,|u_\Gamma|^{\kappa-1}v_\Gamma)\|_{L^2(0,T;\mathbb{L}^2)})
		\end{split}
	\end{align}
	
	Let us estimate the last two terms of \eqref{estimate:mathcal:F}. By {\bf (A2)}, following continuous embedding:
	\begin{align*}
		L^8(Q) \hookrightarrow L^{2\sigma}(Q),\text{ and } L^{8}(\Sigma)\hookrightarrow L^{2\kappa}(\Sigma).
	\end{align*}
	
	Then, we have
	\begin{align}
		\label{estimate:mathcal:F:02} 
		\|(|u|^\sigma,|u_\Gamma|^{\kappa})\|_{L^2(0,T;\mathbb{L}^2)}
		\leq  C \left(\|u\|_{L^8(Q)}^\sigma + \|u\|_{L^8(\Sigma)}^{\kappa} \right) 
	\end{align}
	
	On the other hand, by H\"older's inequality, the last term of the right-hand side can be computed as 
	\begin{align*}
		\|(|u|^{\sigma-1}v,|u_\Gamma|^{\kappa-1}v_\Gamma)\|_{L^2(0,T;\mathbb{L}^2)}
		\leq & \|u^{\sigma-1}\|_{L^{8/3}(Q)} \|v\|_{L^8(Q)} + \|u_\Gamma^{\kappa-1}\|_{L^{8/3}(\Sigma)} \|v_\Gamma\|_{L^8(\Sigma)}\\
		=& \|u\|_{L^{\frac{8}{3}(\sigma-1)}(Q)}^{\sigma-1}\|v\|_{L^8(Q)} + \|u_\Gamma\|_{L^{\frac{8}{3}(\kappa-1)}(\Sigma)}^{\kappa-1}\|v_\Gamma\|_{L^8(\Sigma)}.
	\end{align*}
	
	Once again, {\bf (A2)} implies the continuous embeddings:
	\begin{align}
		\label{Sobolev:embedding:01}
		L^8(Q) \hookrightarrow L^{\frac{8}{3}(\sigma-1)}(Q),\text{ and } L^8(\Sigma) \hookrightarrow L^{\frac{8}{3}(\kappa-1)}(\Sigma).
	\end{align}
	
	Then, for all $(u,u_\Gamma,v,v_\Gamma)\in \mathbb{B}_r$ with $0<r<1$, we have
	\begin{align}
		\nonumber 
		\|(|u|^{\sigma-1}v,|u_\Gamma|^{\kappa-1}v_\Gamma)\|_{L^2(0,T;\mathbb{L}^2)}=&\|u^{\sigma-1}v\|_{L^2(0,T;L^2(\Omega))} + \|u_\Gamma^{\kappa-1}v_\Gamma\|_{L^2(0,T;L^2(\Gamma))}\\
		\nonumber 
		\leq & \|(u,u_\Gamma)\|_{\mathbb{D}}^{\min\{\sigma,\kappa\}-1} \cdot \|(v,v_\Gamma)\|_{\mathbb{D}}\\
		\label{estimate:mathcal:F:03}
		\leq & C \|(u,u_\Gamma,v,v_\Gamma)\|_{\mathbb{D}^2}^{\min\{\sigma,\kappa\}}
	\end{align}
	
	Thus, combining \eqref{estimate:mathcal:F:02} and \eqref{estimate:mathcal:F:03} with \eqref{estimate:mathcal:F}, using the assumption \eqref{assumption:smallness:data} and the fact that $(u,u_\Gamma,v,v_\Gamma)\in \mathbb{B}_r$, we deduce the existence of a constant $C_1>0$ such that 
	\begin{align}
		\label{estimate:F:00}
		\|\mathcal{F}(u,u_\Gamma,v,v_\Gamma)\|_{\mathbb{D}^2}\leq C_1(\epsilon +  \|(u,u_\Gamma, v,v_\Gamma)\|_{\mathbb{D}^2}^{\min\{\sigma,\kappa\}}). 
	\end{align}
	
	Now, let us consider $(u,u_\Gamma,v,v_\Gamma),(\tilde{u},\tilde{u}_\Gamma,\tilde{v},\tilde{v}_\Gamma)\in \mathbb{D}^2$. Then, we have 
	\begin{align}
	\label{estimate:cont:F:00}
	\begin{split} 
		&\|\mathcal{F}(u,u_\Gamma,v,v_\Gamma) -\mathcal{F}(\tilde{u},\tilde{u}_\Gamma,\tilde{v},\tilde{v}_\Gamma)\|_{\mathbb{D}^2}\\
		\leq &C\|(|u|^{\sigma-1}u-|\tilde{u}|^{\sigma-1}u)\|_{L^2(0,T;\mathbb{L}^2)} \\
		&+ C\|(|u|^{\sigma-2}uv-|\tilde{u}|^{\sigma-2}\tilde{u}\tilde{v},|u_\Gamma|^{\kappa-2}u_\Gamma v_\Gamma -|\tilde{u}_\Gamma|^{\kappa-2}\tilde{u}_\Gamma \tilde{v}_\Gamma)\|_{L^2(0,T;\mathbb{L}^2)}.
	\end{split}
	\end{align}
	
	Let us estimate the terms of the right-hand side of \eqref{estimate:cont:F:00}. Firstly, we notice that
	\begin{align}
		\nonumber 
		&\||u|^{\sigma-1}u-|\tilde{u}|^{\sigma-1}\tilde{u}\|_{L^2(Q)}\\
		\nonumber 
		\leq &C\|(|u|^{\sigma-1}+|\tilde{u}|^{\sigma-1})(u-\tilde{u})\|_{L^2(Q)}
		\\
		\nonumber 
		=& C\left( \|u\|_{L^{\frac{8}{3}(\sigma-1)}(Q)}^{\sigma-1}  + C\|\tilde{u}\|_{L^{\frac{8}{3}(\sigma-1)}(Q)}^{\sigma-1}\right)\|u-\tilde{u}\|_{L^8(Q)}\\
		\label{estimate:cont:F:01}
		\leq & C\left( \|u\|_{L^8(Q)}^{\sigma-1} + \|\tilde{u}\|_{L^8(Q)}^{\sigma-1} \right) \|u-\tilde{u}\|_{L^8(Q)},
	\end{align}
	
	where we have used \eqref{Sobolev:embedding:01}. In the same manner, we can easily check that 
	\begin{align*}
		\begin{split} 
		\||u_\Gamma|^{\kappa-1}u_\Gamma-|\tilde{u}_\Gamma|^{\kappa-1}\tilde{u}_\Gamma\|_{L^2(\Sigma)} \leq C\left( \|u_\Gamma\|_{L^8(\Sigma))}^{\kappa-1} + \|\tilde{u}_\Gamma\|_{L^8(\Sigma)}^{\kappa-1} \right)\|u_\Gamma -\tilde{u}_\Gamma\|_{L^8(\Sigma)}.
		\end{split}
	\end{align*}
	
	Secondly, we see that 
	\begin{align}
		\nonumber 
		&\||u|^{\sigma-2}uv-|\tilde{u}|^{\sigma-2}\tilde{u}\tilde{v}\|_{L^2(Q)}\\
		\nonumber 
		=&\||u|^{\sigma-2}u(v-\tilde{v}) + (|u|^{\sigma-2}u-\tilde{u}^{\sigma-2}\tilde{u})\tilde{v}\|_{L^2(Q)}\\
		\label{estimate:cont:F03}
		\leq & C \||u|^{\sigma-1}\|_{L^{8/3}(Q)} \|v-\tilde{v}\|_{L^8(Q)} + C\|(|u|^{\sigma-2}u-|\tilde{u}|^{\sigma-2}\tilde{u})\tilde{v}\|_{L^2(Q)}^2.
	\end{align}
	
	Now, by interpolation of $L^p$ spaces and assumption {\bf (A2)}, it is easy to see that 
	\begin{align}
		\label{estimate:cont:F04}
		\||u|^{\sigma-1}\|_{L^{8/3}(Q)} \leq C\|u\|_{L^8(Q)}^{\sigma-1}.
	\end{align}
	
	Moreover, similar to \eqref{estimate:cont:F:01} and using H\"older's inequality we get
	\begin{align}
		\nonumber 
		&\|(|u|^{\sigma-2}u-|\tilde{u}|^{\sigma-2}\tilde{u})\tilde{v}\|_{L^2(Q)}\\
		\nonumber 
		\leq & C \|(|u|^{\sigma-2}+|\tilde{u}|^{\sigma-2})\tilde{v}(u-\tilde{u})\|_{L^2(Q)}\\
		\nonumber 
		\leq & C(\||u|^{\sigma-2}\|_{L^4(Q)} +\||\tilde{u}|^{\sigma-2}\|_{L^4(Q)}\|)\|\tilde{v}\|_{L^8(Q)}\|u-\tilde{u}\|_{L^8(Q)}\\
		\label{estimate:cont:F05}
		\leq & C \left( \|u\|_{L^8(Q)}^{\sigma-2} + \|\tilde{u}\|_{L^8(Q)}^{\sigma-2} \right) \|\tilde{v}\|_{L^8(Q)} \|u-\tilde{u}\|_{L^8(Q)}. 
	\end{align}
	where we have used the fact that 
	\begin{align*}
		L^8(Q)\hookrightarrow L^{4(\sigma-2)}(Q)
	\end{align*}
	by assumption {\bf (A2)}. Then, combining \eqref{estimate:cont:F04} and \eqref{estimate:cont:F05} with \eqref{estimate:cont:F03}, we have that  
	\begin{align}
	\label{estimate:cont:F06}
	\begin{split} 
		&\||u|^{\sigma-2} uv-|\tilde{u}|^{\sigma-2} \tilde{u}\tilde{v}\|_{L^2(Q)}\\
		\leq &  C( \|u\|_{L^8(Q)}^{\sigma-2} + \|\tilde{u}\|_{L^8(Q)}^{\sigma-2} )\|\tilde{v}\|_{L^8(Q)} \|u-\tilde{u}\|_{L^8(Q)}+C\|u\|_{L^8(Q)}^{\sigma-1}\|v-\tilde{v}\|_{L^8(Q)}
	\end{split}
	\end{align}
	
	Analogously, we have the following estimate
	\begin{align}
		\label{estimate:cont:F07}
		\begin{split}
			&\||u_\Gamma|^{\kappa-2}u_\Gamma v_\Gamma -|\tilde{u}_\Gamma|^{\kappa-2}\tilde{u}_\Gamma \tilde{v}_\Gamma\|_{L^2(\Sigma)}\\
			\leq & C\left(\|u_\Gamma\|_{L^8(\Sigma)}^{\kappa-2} + \|\tilde{u}_\Gamma\|_{L^8(\Sigma)}^{\kappa-2}  \right) \|\tilde{v}_\Gamma\|_{L^8(\Sigma)} \|u_\Gamma-\tilde{u}_\Gamma\|_{L^8(\Sigma)}\\ 
			&+C\|u_\Gamma\|_{L^8(\Sigma)}^{\kappa-1} \|v_\Gamma-\tilde{v}_\Gamma\|_{L^8(\Sigma)}
		\end{split}
	\end{align}
	
	Then, combining \eqref{estimate:cont:F06} and \eqref{estimate:cont:F07} with \eqref{estimate:cont:F:00} and using the fact that $(u,u_\Gamma,v,v_\Gamma),(\tilde{u},\tilde{u}_\Gamma,\tilde{v},\tilde{v}_\Gamma)$ belongs to $\mathbb{B}_r$ (with $0<r<1$), we deduce the existence of a constant $C_2>0$ such that 
	\begin{align*}
		\begin{split}
			&\|\mathcal{F}(u,u_\Gamma,v,v_\Gamma)-\mathcal{F}(\tilde{u},\tilde{u}_\Gamma,\tilde{v},\tilde{v}_\Gamma)\|_{\mathbb{D}^2}\\
			\leq &C_2 r^{\min\{\sigma,\kappa\}-1}\|(u,u_\Gamma,v,v_\Gamma)-(\tilde{u},\tilde{u}_\Gamma,\tilde{v},\tilde{v}_\Gamma)\|_{\mathbb{D}^2}.
		\end{split}
	\end{align*}
	
	To conclude, let us choose $0<r<1$ and $\epsilon>0$ such that 
	\begin{align}
		\label{estimate:cont:F09}
		r^{\min\{\sigma,\kappa\}-1}\leq \min \left\{\dfrac{1}{2C_1},\dfrac{1}{C_2} \right\}<1, \quad \epsilon \leq \dfrac{1}{2C_1}r.
	\end{align}
	
	Then, using the conditions \eqref{estimate:cont:F09} in \eqref{estimate:F:00} and \eqref{estimate:cont:F09} implies that $\mathcal{F}:\mathbb{B}_r\to \mathbb{B}_r$ is a contraction mapping. Thus, by Banach's fixed point Theorem, the equation 
	\begin{align*}
		\mathcal{F}(u,u_\Gamma,v,v_\Gamma)=(u,u_\Gamma,v,v_\Gamma)
	\end{align*}
	has a unique solution $(u,u_\Gamma,v,v_\Gamma)\in \mathbb{B}_r$. This concludes the proof of the Proposition \ref{prop:existence:nonlinear}. 
\end{proof}

\section{Carleman inequalities for parabolic systems with dynamic boundary conditions}
\label{section:Carleman:inequalities}
In this section, we prove new Carleman estimates for parabolic systems related to the existence of insensitizing controls of the functional $\Phi$ defined in \eqref{def:Phi:insen}.


To fix ideas, we introduce the adjoint system
\begin{align}
	\label{adjoint:system:psi}
	\begin{cases}
		L(\psi) =g^1&\text{ in }Q,\\
		L_\Gamma(\psi,\psi_\Gamma)=g_\Gamma^1&\text{ on }\Sigma,\\
		\psi\big|_\Gamma=\psi_\Gamma &\text{ on }\Sigma,\\
		(\psi(\cdot,0),\psi_\Gamma(\cdot,0))=(\psi^0,\psi_\Gamma^0)&\text{ in }\Omega\times \Gamma,
	\end{cases}
\end{align}
\begin{align}
	\label{adjoint:system:varphi}
	\begin{cases}
		L^*(\varphi) =g^0+\mathbbm{1}_\mathcal{O} \psi &\text{ in }Q,\\
		L_\Gamma^*(\varphi,\varphi_\Gamma) =g_\Gamma^0+\nabla_\Gamma \cdot (\mathbbm{1}_{\mathcal{G}} \nabla_\Gamma \psi_\Gamma)&\text{ on }\Sigma,\\
		\varphi\big|_\Gamma=\varphi_\Gamma&\text{ on }\Sigma,\\
		(\varphi(\cdot,T),\varphi_\Gamma(\cdot,T))=(0,0)&\text{ in }\Omega\times \Gamma.
	\end{cases}
\end{align}

Then, the main task of this section is to prove an inequality of the form
\begin{align}
	\label{observability:ineq}
	\begin{split}
	&\iint_{Q} \rho_1(t)(|\varphi|^2 + |\nabla \varphi|^2 +| \psi|^2 + |\nabla \psi|^2)\,dx\,dt\\
	 &+ \iint_\Sigma \rho_2(t)(|\varphi_\Gamma|^2 + |\nabla_\Gamma \varphi_\Gamma|^2 +|\psi_\Gamma|^2 +|\nabla_\Gamma \varphi|^2)\,dS\,dt\\
	\leq & C \iint_Q \rho_3(t) (|g^0|^2+ |g^1|^2)\,dx\,dt +C \iint_\Sigma \rho_4(t)(|g_\Gamma^0|^2 + |g_\Gamma^1|^2)\,dS\,dt\\
	&+C\iint_{Q_\omega} \rho_5(t) |\varphi|^2\,dx\,dt, 
	\end{split} 
\end{align} 
for each solution of  \eqref{adjoint:system:psi}-\eqref{adjoint:system:varphi}, where $\rho_i$ $(i=1,\ldots,5)$ are weight functions which blow up at $t=0$ and $C>0$ only depends on $\Omega,\omega,\mathcal{O},\mathcal{G}$ and $T$.

We emphasize that \eqref{observability:ineq} only depends on the local integral terms of the variables $(\varphi,\varphi_\Gamma)$ on the right-hand side. 

\subsection{First Carleman estimates for the heat equation with dynamic boundary conditions}

Now, we will prove a Carleman estimate for parabolic equations with dynamic boundary conditions with source terms in $L^2(0,T;\mathbb{L}^2)$. More precisely, we will consider the system

\begin{align}
		\label{heat:eq:zeta:linear}
		\begin{cases}
			L^*(\zeta) =h&\text{ in }Q,\\
			L_\Gamma^*(\zeta,\zeta_\Gamma)=h_\Gamma &\text{ on }\Sigma,\\
			\zeta\big|_\Gamma=\zeta_\Gamma &\text{ on }\Sigma,\\
			(\zeta(\cdot,T),\zeta_\Gamma (\cdot,T))=(\zeta^T,\zeta_\Gamma^T) &\text{ in }\Omega\times \Gamma,
		\end{cases}
\end{align}

We use the classical weight functions used in Carleman estimates for the heat equation introduced by Fursikov and Imanuvilov. For each $\omega \Subset \Omega$, define $\eta^0\in C^2(\overline{\Omega})$ with the following requirements:
\begin{align*}
\eta^0>0\text{ in }\Omega,\quad \eta^0=0\text{ on }\partial\Omega\quad \text{ and }|\nabla \eta^0|\geq C>0\text{ in }\overline{\Omega\setminus \omega}.
\end{align*}
The existence of $\eta^0$ is guaranteed see for instance \cite{fursikov1996imanuvilov}. Now, we define the weight functions
\begin{align}
	\label{def:weight:alpha:xi}
	\xi(x,t):=\dfrac{e^{\lambda \eta^0(x)}}{t(T-t)}\quad \text{ and }\quad  \alpha(x,t):=\dfrac{e^{2\lambda \|\eta^0\|_\infty}-e^{\lambda \eta^0(x)}}{t(T-t)}\quad \forall (x,t)\in \overline{\Omega}\times (0,T).
\end{align}

For $m\geq 0$, $(\zeta,\zeta_\Gamma)\in \mathbb{E}$ and $s,\lambda>0$ we shall use the notation
\begin{align*}
		I(m,\zeta,\zeta_\Gamma)
		:=&\iint_Q e^{-2s\alpha} (s^{m+3} \lambda^{m+4} \xi^{m+3} |\zeta|^2 + s^{m+1} \lambda^{m+2} \xi^{m+1} |\nabla \zeta|^2) \,dx\,dt\\
		&+\iint_Q e^{-2s\alpha} (s^{m-1} \lambda^{m} \xi^{m-1} (|\pt \zeta|^2 +|\Delta \zeta|^2  ))\,dx\,dt\\
			&+\iint_\Sigma e^{-2s\alpha}(s^{m+3}\lambda^{m+3}\xi^{m+3} |\zeta_\Gamma|^2+s^{m+1}\lambda^{m+1} \xi^{m+1}(|\pnu \zeta|^2 +|\nabla_\Gamma \zeta_\Gamma|^2) )\,dS\,dt\\
			 &+s^{m-1}\lambda^m \iint_\Sigma e^{-2s\alpha}\xi^{m-1}(|\pt \zeta|^2 +|\Delta_\Gamma \zeta_\Gamma|^2)\,dS\,dt
\end{align*}

We can now state the Carleman estimate for the heat equation with dynamic boundary conditions with source terms in $L^2(0,T;\mathbb{L}^2)$.

\begin{proposition}
	\label{proposition:maniar}
	Let $\omega\Subset \Omega$ and $\omega'\subset \Omega$ such that $\omega'\Subset \omega$. Define $\eta^0$, $\alpha$ and $\xi$ associated to $\omega'$ as in \eqref{def:weight:alpha:xi}. Then, there exist constants $C_1,\lambda_0,s_0>0$ such that for all $\lambda \geq \lambda_0$ and $s\geq s_0$, we have 
	\begin{align}
		\label{Carleman:m:heat:eq}
		\begin{split}
			 I(m,\zeta,\zeta_\Gamma) \leq & C_1 s^m \lambda^m \iint_Q e^{-2s\alpha} \xi^m |h|^2 dxdt + C_1s^m\lambda^m \iint_\Sigma e^{-2s\alpha}\xi^m |h_\Gamma|^2 dSdt\\
			&C_1 s^{m+3}\lambda^{m+4}\iint_{Q_\omega} e^{-2s\alpha} \xi^{m+3}|\zeta|^2 dxdt,
		\end{split}
	\end{align}
	where each $(\zeta,\zeta_\Gamma)$ solution of \eqref{heat:eq:zeta:linear} with $(h,h_\Gamma)\in L^2(0,T;\mathbb{L}^2)$, $(R,R_\Gamma)\in L^\infty(0,T;\mathbb{L}^\infty)$ and final data  $(\zeta^T,\zeta_\Gamma^T)\in \mathbb{L}^2$.
\end{proposition}

If $m=0$, \eqref{Carleman:m:heat:eq} is just the Carleman estimate for the heat equation with dynamic boundary conditions previously obtained in \cite{maniar2017null}. If $m>0$, we apply \eqref{Carleman:m:heat:eq} (with $m=0$) to the pair $(\xi^{m/2}\zeta,\xi^{m/2}\zeta_\Gamma)$ and take $s$ and $\lambda$ large enough to absorb the remainder terms.

\subsection{A Carleman estimate for the heat equation in Sobolev spaces of negative order}

Now, we restrict our attention to deduce a Carleman inequality for a heat equation with dynamic boundary conditions with source terms in $L^2(0,T;(\mathbb{H}^1)')$.

Let $(\phi,\phi_\Gamma)$ be the solution of 
	\begin{align}
		\label{system:phi:phi:gamma}
		\begin{cases}
			L^*(\phi) =h^0+\nabla\cdot h^1&\text{ in }Q,\\
			L^*_\Gamma(\phi,\phi_\Gamma) =h_\Gamma^0+\nabla_\Gamma \cdot h_\Gamma^1&\text{ on }\Sigma,\\
			\phi\big|_\Gamma=\phi_\Gamma&\text{ on }\Sigma,\\
			(\phi,\phi_\Gamma)(\cdot,T)=(\phi^T,\phi_\Gamma^T)&\text{ in }\Omega\times \Gamma.
		\end{cases}
	\end{align}

We point out that for $(\phi^T,\phi_\Gamma^T)\in \mathbb{L}^2$, $(h^0,h_\Gamma^0)\in L^2(0,T;\mathbb{L}^2)$ and $(h^1,h_\Gamma^1)\in L^2(0,T;\mathbb{L}^2)$, the solution $(\phi,\phi_\Gamma) $ of \eqref{system:phi:phi:gamma} exists in the distributional sense. Moreover, $(\phi,\phi_\Gamma)\in C^0([0,T];L^2)$ and there exists a constant $C>0$ such that
\begin{align*}
	\|(\phi,\phi_\Gamma)\|_{C^0([0,T];\mathbb{L}^2)}\leq C (\|(\phi^T,\phi_\Gamma^T)\|_{\mathbb{L}^2}+ \|(h^0,h_\Gamma^0)\|_{L^2(0,T;\mathbb{L}^2)} + \|(h^1,h_\Gamma^1)\|_{L^2(0,T;\mathbb{L}^2)}).
\end{align*}

For $\omega \Subset \Omega$, define the weight functions $\xi$ and $\alpha$ in \eqref{def:weight:alpha:xi}. Besides, for $n\geq 0$, $(\phi,\phi_\Gamma)\in L^2(0,T;\mathbb{H}^1)$ and $s,\lambda>0$, we write 
\begin{align*}
	\begin{split}
		J(n,\phi,\phi_\Gamma)&:=\iint_Q e^{-2s\alpha} (s^{n+3}\lambda^{n+4}\xi^{n+3} |\phi|^2 + s^{n+1}\lambda^{n+2} \xi^{n+1} |\nabla \phi|^2)dxdt\\
		 &+ \iint_\Sigma e^{-2s\alpha} (s^{n+3}\lambda^{n+3}\xi^{n+3} |\phi_\Gamma|^2+s^{n+1}\lambda^{n+2} \xi^{n+1} |\nabla_\Gamma \phi_\Gamma|^2 )dxdt\\
	\end{split}
\end{align*}

Additionally, we introduce the spaces
\begin{align*}
	X:=\{r\in [L^2(\Omega)]^d \,:\, r\cdot \nu \in L^2(\Gamma) \},\,\text{ and } Y:=[L^2(\Gamma)]^d.
\end{align*}

We point out that $H^1(\Omega) \subsetneq X$.

\begin{lemma}
	\label{Lemma:Carleman:phi:phi:gamma}
	Let $\omega\Subset \Omega$ and $\omega'\Subset \omega$ and define the weights $\alpha$ and $\xi$ in \eqref{def:weight:alpha:xi} associated to $\omega'$. 
	Moreover, let $(h^0,h_\Gamma^0)\in L^2(0,T;\mathbb{L}^2)$, $(h^1,h_\Gamma^1)\in L^2(0,T;X\times Y)$, $(\phi^T,\phi_\Gamma^T)\in \mathbb{L}^2$ and $(R,R_\Gamma)\in L^\infty(0,T;\mathbb{L}^\infty)$.
	Then, there exist constants $C,s_1,\lambda_1>0$ such that for all $\lambda \geq\lambda _1$ and $s\geq s_1$, we have	\begin{align}
		\label{Carleman:phi:phi:gamma}
		\begin{split}
			J(n,\phi,\phi_\Gamma)
			\leq & C\iint_Q e^{-2s\alpha } (\lambda^n s^n \xi^n |h^0|^2 +s^{n+2}\lambda^{n+2} \xi^{n+2} |h^1|^2)\,dx\,dt \\
			&+ C\iint_\Sigma e^{-2s\alpha}(s^n\lambda^{n+1}\xi^n|h_\Gamma^0|^2 + s^{n+2}\lambda^{n+2} \xi^{n+2} |h_\Gamma^1|^2 + s^n\lambda^{n+1}\xi^{n} |h^1\cdot\nu|^2)\,dS\,dt\\
			&+Cs^{n+3}\lambda^{n+4}\iint_{Q_\omega} e^{-2s\alpha} \xi^{n+3} |\phi|^2 \,dx\,dt, 
		\end{split}
	\end{align}
	for each solution $(\phi,\phi_\Gamma)$ of \eqref{system:phi:phi:gamma}.
\end{lemma}

We point out that \eqref{Carleman:phi:phi:gamma} coincides with the Carleman estimate obtained in \cite{Boutaayamou2021Thecost} (see also \cite{Boutaayamou2022Stackelberg}) in the case $n=0$. To deduce the general case $n>0$, we apply \eqref{Carleman:phi:phi:gamma} to the couple $(\xi^{n/2}\phi,\xi^{n/2}\phi_\Gamma)$ and absorb the rest of the terms taking $s$ and $\lambda$ large enough.


\subsection{A new Carleman estimate for the adjoint coupled system}

In this section, we prove a Carleman estimate for the adjoint system \eqref{adjoint:system:psi}-\eqref{adjoint:system:varphi}.

\begin{theorem}
	Let $(g^0,g_\Gamma^0), (g^1,g_\Gamma^1)\in L^2(0,T;\mathbb{L}^2)$. Assume {\bf (H1)} and consider $\omega''\Subset \omega\cap \mathcal{O}$ to define $\eta^0,\alpha$ and $\xi$ with respect to $\omega''$. Then, there exist constants $C_3,\lambda_3,s_3>0$ such that for all $\lambda\geq \lambda_3$ and $s\geq s_3$, we have
	\begin{align}
		\label{Carleman:coupled:phi:psi}
		\begin{split}
			&I(1,\psi,\psi_\Gamma)+J(0,\varphi,\varphi_\Gamma)\\
			\leq & C_3 \iint_Q e^{-2s\alpha} (s^4\lambda^5 \xi^4 |g^0|^2 + s\lambda \xi^5 |g^1|^2) \,dx\,dt + C_3\iint_\Sigma e^{-2s\alpha}(\lambda |g_\Gamma^0|^2 + s\lambda \xi^2 |g_\Gamma^1|^2)\,dS\,dt\\
			&+C_3 s^9\lambda^{11}\iint_{Q_\omega} e^{-2s\alpha} \xi^9 |\varphi|^2 \,dx\,dt,
		\end{split}
	\end{align}
	for all $(\varphi,\varphi_\Gamma,\psi,\psi_\Gamma)$ solution of \eqref{adjoint:system:psi}-\eqref{adjoint:system:varphi} with $(g^0,g_\Gamma^0),(g^1,g_\Gamma^1)\in L^2(0,T;\mathbb{L}^2)$ and $(\psi^0,\psi_\Gamma^0)\in \mathbb{L}^2$. 
\end{theorem}

\begin{proof}
	We divide the proof into three steps:
	
	\noindent $\bullet$ \textit{ Step 1: Global estimates for $(\psi,\psi_\Gamma)$.} In this step, we perform an initial estimate for the variables $(\psi,\psi_\Gamma)$. To do this, we apply the Proposition \ref{proposition:maniar} to $(\psi,\psi_\Gamma)$ in  \eqref{adjoint:system:psi}:
	\begin{align}
		\label{Carleman:estimate:global:01}
		\begin{split} 
		&I(1,\psi,\psi_\Gamma)\\
		\leq & Cs\lambda \iint_{Q} e^{-2s\alpha}\xi |g^1|^2\,dx\,dt + Cs\lambda \iint_\Sigma e^{-2s\alpha}\xi |g_\Gamma ^1|^2\,dS\,dt\\
		&+Cs^4\lambda^5 \iint_{Q_{\omega''}} e^{-2s\alpha}\xi^4|\psi|^2\,dx\,dt\quad \forall s\geq s_1,\, \forall \lambda\geq\lambda_1.
		\end{split}
	\end{align}
	Our next task is to estimate the local term of $\psi$ in the right-hand side of \eqref{Carleman:estimate:global:01}.
	
	\noindent $\bullet$ \textit{ Step 2: Estimates of the local term of $\psi$.} Let $\omega'\subset \Omega$ such that $\omega''\Subset \omega'\Subset \omega\cap \mathcal{O}$. Now, define a function $\theta \in C^\infty(\overline{\Omega})$ such that 
	\begin{align*}
		\theta=1\text{ in }\omega'',\quad 0\leq \theta\leq 1\text{ in }\omega'\quad \text{ and }\theta\equiv 0\text{ in }\Omega\setminus \overline{\omega'}.
	\end{align*}
	
	In addition, we assume that $\theta$ satisfies 
	\begin{align}
		\label{assumptions:theta}
		\dfrac{|\nabla \theta|}{\theta^{1/2}} \in L^\infty(\omega')\quad \text{ and }\dfrac{\Delta \theta}{\theta^{1/2}}\in L^\infty(\omega').
	\end{align}
	
	Then, using the equation \eqref{adjoint:system:varphi} and integrating by parts in time and space, we see that 
	\begin{align}
		\nonumber 
		&s^4\lambda^5 \iint_{Q_{\omega'}} e^{-2s\alpha} \xi^4\theta |\psi|^2\,dx\,dt\\
		\nonumber 
		=&s^4\lambda^5\iint_{Q_{\omega'}} e^{-2s\alpha} \xi^4 \theta g^1 \varphi\,dx\,dt -s^4\lambda^5 \iint_{Q_{\omega}'}e^{-2s\alpha} \xi^4 \theta \psi g^0\,dx\,dt \\
		\nonumber 
		&+s^4\lambda^5\iint_{Q_{\omega'}} \pt(e^{-2s\alpha} \xi^4 \theta)\psi \varphi\,dx\,dt+2s^4\lambda^5 \iint_{Q_{\omega'}} \nabla \left(e^{-2s\alpha}\xi^4 \theta \right) \cdot \nabla \psi \varphi \,dx\,dt\\
		\nonumber 
		&+s^4\lambda^5 \iint_{Q_\omega'}\Delta (e^{-2s\alpha} \xi^4 \theta)\psi \varphi\,dx\,dt\\
		\label{Carleman:estimate:global:02}
		=& I_1+I_2+I_3+I_4+I_5.
	\end{align}
	
	Now, we focus on estimating the terms $I_i$, $1\leq i\leq 5$. Firstly, we notice that 
	\begin{align*}
		I_1\leq Cs^8\lambda^{10} \iint_{Q_{\omega'}} e^{-2s\alpha} \xi^8 \theta |\varphi|^2\,dx\,dt + C\iint_{Q_{\omega'}} e^{-2s\alpha} \theta |g^1|^2\,dx\,dt.
	\end{align*}
	
	Secondly, by Young's inequality, for each $\delta_1>0$, there exists $C(\delta_1)>0$ such that 
	\begin{align*}
		I_2\leq C(\delta^1)s^4\lambda^5 \iint_{Q_{\omega'}} e^{-2s\alpha} \xi^4 \theta |g^0|^2\,dx\,dt +\dfrac{\delta_1}{2} s^4\lambda^5 \iint_{Q_{\omega'}} e^{-2s\alpha} \xi^4 \theta |\psi|^2\,dx\,dt.
	\end{align*}
	
	To estimate $I_3$, we take into account that $|\pt \alpha|\leq C\lambda \xi^2\text{ and }|\pt \xi|\leq C\lambda\xi^2$. In particular, this implies that $|\pt (e^{-2s\alpha}\xi^4 \theta )|\leq Cs\lambda e^{-2s\alpha}\xi^6$. Then for all $\delta_2>0$, there exists $C(\delta_2)>0$ such taht 
	\begin{align*}
		I_3\leq C(\delta_2) s^6\lambda^7 \iint_{Q_{\omega'}} e^{-2s\alpha} \xi^4 \theta |\varphi|^2\,dx\,dt +\dfrac{\delta_2}{2}s^4\lambda^5 \iint_{Q_{\omega'}} e^{-2s\alpha} \xi^4 \theta |\psi|^2\,dx\,dt.
	\end{align*}
	
	Besides, $I_4$ can be estimated in the following way
	\begin{align*}
		I_4\leq & Cs^9\lambda^{11} \iint_{Q_{\omega'}} e^{-2s\alpha} \xi^9 \theta |\varphi|^2\,dx\,dt +Cs^7\lambda^9 \iint_{Q_{\omega'}} e^{-2s\alpha} \xi^7 \dfrac{|\nabla \theta|^2}{\theta} |\varphi|^2\,dx\,dt\\
		&+Cs\lambda \iint_{Q_{\omega'}} e^{-2s\alpha} \xi \theta |\nabla \psi|^2\,dx\,dt.
	\end{align*}
	
	Now, notice that 
	\begin{align*}
		|\Delta (e^{-2s\alpha} \xi^4 \theta)|\leq Cs^2\lambda^2 e^{-2s\alpha} \xi^6 \theta +Cs\lambda e^{-2s\alpha} \xi^5 |\nabla \theta| + e^{-2s\alpha} \xi^4 |\Delta \theta|.
	\end{align*}
	
	Then, for all $\delta_3>0$, there exists $C(\delta_3)>0$ such that 
	\begin{align*}
		I_5 \leq & C(\delta_3) s^8 \lambda^9 \iint_{Q_{\omega'}} e^{-2s\alpha} \xi^8 \theta |\varphi|^2 \,dx\,dt 
		+C(\delta_3) s^6 \lambda^7\iint_{Q_{\omega'}} e^{-2s\alpha} \xi^6 \frac{|\nabla \theta|^2}{\theta} |\varphi|^2\,dx\,dt
		\\
		&+ C(\delta_3) s^4\lambda^5 \iint_{Q_{\omega'}} e^{-2s\alpha} \xi^4 \dfrac{|\Delta \theta|^2}{\theta} |\varphi|^2\,dx\,dt + \dfrac{\delta_3}{2} s^4 \lambda^5 \iint_{Q_{\omega'}} e^{-2s\alpha} \xi^4\theta |\psi|^2\,dx\,dt.
	\end{align*}
	
	Then, combining these estimates with \eqref{Carleman:estimate:global:02}, using the assumption \eqref{assumptions:theta}, choosing the $\delta_i's$ such that $\sum_{i=1}^3 \delta_i=1$, and using the fact that $\theta\equiv 1$ in $\omega''$ we obtain
	\begin{align}
		\label{Carleman:estimate:global:03}
		\begin{split} 
		&s^4\lambda^5\iint_{Q_{\omega''}} e^{-2s\alpha} \xi^4 |\psi|^2\,dx\,dt\\
		\leq & Cs^9\lambda^{11} \iint_{Q_{\omega''}} e^{-2s\alpha} \xi^9 \theta |\varphi|^2\,dx\,dt + C\iint_{Q_{\omega'}} e^{-2s\alpha} \theta |g^1|^2\,dx\,dt\\ 
		&+ Cs^4\lambda^5 \iint_{Q_{\omega'}} e^{-2s\alpha} \xi^4 \theta |g^0|^2\,dx\,dt +Cs\lambda \iint_{Q_{\omega'}} e^{-2s\alpha} \xi |\nabla \psi|^2\,dx\,dt.
		\end{split}
	\end{align} 
		
	Now, combining \eqref{Carleman:estimate:global:03} with \eqref{Carleman:estimate:global:01} and taking $s$ and $\lambda$ large enough, we see that 
	\begin{align}
		\label{Carleman:estimate:global:04}
		\begin{split}
			&I(1,\psi,\psi_\Gamma)\\
			\leq & Cs^4\lambda^5 \iint_{Q} e^{-2s\alpha} \xi^4 |g^0|^2\,dx\,dt +Cs\lambda \iint_{Q} e^{-2s\alpha} \xi |g^1|^2\,dx\,dt\\
			&+Cs\lambda \iint_{\Sigma} e^{-2s\alpha}\xi |g_\Gamma^1|^2\,dS\,dt + Cs^9\lambda^{11} \iint_{Q_{\omega'}} e^{-2s\alpha} \xi^9 |\varphi|^2\,dx\,dt
		\end{split}
	\end{align}
	for all $s\geq s_2$ and $\lambda\geq \lambda_2$. 
	
	\noindent $\bullet$ \textit{ Step 3: Global estimates for $(\varphi,\varphi_\Gamma)$.} Now, we apply the Lemma \ref{Lemma:Carleman:phi:phi:gamma} to $(\varphi,\varphi_\Gamma)$ solution of \eqref{adjoint:system:varphi}:
	 \begin{align}
	 	\label{Carleman:estimate:global:05}
	 	\begin{split} 
	 	&J(0,\varphi,\varphi_\Gamma)\\
	 	\leq & C\iint_{Q} e^{-2s\alpha} (|g^0|^2+ \mathbbm{1}_{\mathcal{O}} |\psi|^2)\,dx\,dt + C\iint_{\Sigma} e^{-2s\alpha} (\lambda |g_\Gamma^0|^2 + s^2\lambda^2 \xi^2 \mathbbm{1}_{\mathcal{G}} |\nabla_\Gamma \psi_\Gamma|^2 )\,dS\,dt\\
	 	&+Cs^3\lambda^4 \iint_{Q_{\omega''}} e^{-2s\alpha} \xi^3 |\varphi|^2\,dx\,dt.
	 	\end{split} 
	 \end{align}
	 
	 Then, combining \eqref{Carleman:estimate:global:04} and \eqref{Carleman:estimate:global:05}, we have 
	 \begin{align}
	 \label{Carleman:estimate:global:06}
	 	\begin{split}
	 	&J(0,\varphi,\varphi_\Gamma)\\
	 	\leq & C\iint_{Q} e^{-2s\alpha}(s^4\lambda^5 \xi^4 |g^0|^2+s\lambda \xi |g^1|^2)\,dx\,dt + C\iint_{\Sigma} e^{-2s\alpha} (\lambda |g_\Gamma^0|^2 + s\lambda |g_\Gamma^1|^2)\,dS\,dt\\
	 	&+Cs^9\lambda^{10}\iint_{Q_{\omega'}} e^{-2s\alpha} \xi^9 |\varphi|^2\,dx\,dt,\quad \forall s\geq s_2,\, \forall \lambda\geq \lambda_2,
	 	\end{split} 
	 \end{align}
	 for all solutions of the adjoint system \eqref{adjoint:system:psi}-\eqref{adjoint:system:varphi}. Finally, adding \eqref{Carleman:estimate:global:04} and \eqref{Carleman:estimate:global:06}, taking $s_3=\max\{s_1,s_2\}$ and $\lambda_3:=\max\{\lambda_1,\lambda_2\}$ we obtain \eqref{Carleman:coupled:phi:psi}.
\end{proof}

\section{Null controllability of the linear system}
\label{section:Null:controllability}
In this section, we devote to proving the null controllability of the linear coupled system \eqref{coupled:linear:system:y}-\eqref{coupled:linear:system:z}. 

\subsection{Observability inequality}

We start with a suitable observability inequality with weights that blow up only at $t\to 0$. Consider
\begin{align*}
	l(t):=\begin{cases}
		t(T-t) &\text{ if }0\leq t\leq T/2,\\
		T^2/4 &\text{ if }T/2<t\leq T.
	\end{cases}
\end{align*}

Now, define the new weight functions 
\begin{align*}
	\beta(x,t):=\dfrac{e^{2\lambda \|\eta^0\|_\infty}-e^{\lambda \eta^0(x)}}{l(t)},\quad \gamma(x,t):=\dfrac{e^{\lambda \eta^0(x)}}{l(t)},\quad \forall (x,t)\in \overline{\Omega}\times (0,T).
\end{align*} 

Moreover, we set 
\begin{align*}
	\beta^*(t):=\max_{x\in \overline{\Omega}} \beta(x,t),\quad \gamma^*(t):=\min_{x\in \overline{\Omega}} \gamma(x,t),\quad t\in (0,T)\\
	\hat{\beta}(t):=\min_{x\in \overline{\Omega}} \beta(x,t),\quad \hat{\gamma}(t):=\max_{x\in \overline{\Omega}}\gamma(x,t),\quad t\in (0,T).
\end{align*}

Then, we have the following observability inequality with weight functions that blow up only at $t=0$.
 
\begin{proposition}
\label{proposition:observability}
	Consider the condition {\bf (H1)} and define the weights $\beta$ and $\gamma$ corresponding to $\omega''\Subset \omega\cap \mathcal{O}$. Then, there exists a constant $C_{obs}>0$ such that
	\begin{align}
		\label{observability:varphi:psi}
		\begin{split}
			&\iint_Q e^{-2s\beta^*} ((\gamma^*)^3 |\varphi|^2+\gamma^* |\nabla \varphi|^2 + (\gamma^*)^4 |\psi|^2 + (\gamma^*)^2 |\nabla \psi|^2)\,dx\,dt\\
			&+\iint_\Sigma e^{-2s\beta^*} ((\gamma^*)^3 |\varphi_\Gamma|^2 + \gamma^* |\nabla_\Gamma \varphi_\Gamma|^2 + (\gamma^*)^4 |\psi_\Gamma|^2 + (\gamma^*)^2 |\nabla \psi_\Gamma|^2 )\,dS\,dt\\
			\leq & C_{\text{obs}} \iint_{Q} e^{-2s\hat{\beta}} (\hat{\gamma}^4 |g^0|^2 + \hat{\gamma}|g^1|^2)\,dx\,dt + C_{\text{obs}}\iint_\Sigma e^{-2s\hat{\beta}}(|g_\Gamma^0|^2 + |g_\Gamma^1|^2)\,dS\,dt\\
			 &+ C_{\text{obs}}\iint_{Q_\omega} e^{-2s\hat{\beta}} \hat{\gamma}^9 |\varphi|^2\,dx\,dt 
		\end{split}
	\end{align}
	for all solutions of \eqref{adjoint:system:psi}-\eqref{adjoint:system:varphi} with $(g^0,g_\Gamma^0),(g^1,g_\Gamma^1)\in L^2(',T;\mathbb{L}^2)$ and $(\psi^0,\psi_\Gamma^0)\in \mathbb{L}^2$.
\end{proposition} 

\begin{proof}
	Firstly, using the cut-off function $\chi \in C^2([0,T])$ such that $0\leq \chi \leq 1$ and 
	\begin{align*}
		\chi(t)=\begin{cases}
			0&\text{ if }t\in [0,T/4],\\
			1&\text{ if }t\in [T/2,T].
		\end{cases}
	\end{align*}
	it is easy to see that each solution $(\varphi,\varphi_\Gamma,\psi,\psi_\Gamma)$ of \eqref{adjoint:system:psi}-\eqref{adjoint:system:varphi} satisfies 
	\begin{align}
		\label{observability:01}
		\begin{split}
			&\|(\varphi,\varphi_\Gamma)\|_{L^\infty(T/2,T;\mathbb{L}^2)}^2 + \|(\varphi,\varphi_\Gamma)\|_{L^2(T/2,T;\mathbb{H}^1)}^2 +\|(\psi,\psi_\Gamma)\|_{L^\infty(T/2,T;\mathbb{L}^2)}^2\\
			 &+ \|(\psi,\psi_\Gamma)\|_{L^2(T/2,T;\mathbb{H}^1)}^2\\
			\leq & C\|(g^0,g_\Gamma^0)\|_{L^2(T/2,T;\mathbb{L}^2)}^2 + C\|(g^1,g_\Gamma^1)\|_{L^2(T/4,T/2;\mathbb{L}^2)}^2+ C\|(\varphi,\varphi_\Gamma)\|_{L^2(T/4,T/2;\mathbb{L}^2)}^2\\
			&+C\|(\psi,\psi_\Gamma)\|_{L^2(T/4,T/2;\mathbb{L}^2)}^2.
		\end{split}
	\end{align}
	
	Now, we focus on writing the Carleman estimate \eqref{Carleman:coupled:phi:psi} in terms of the new weights.
	
	Firstly, since $\alpha=\beta$ and since $l(t)=T^2/4$ in $[T/2,T]$ we easily see that the right-hand side of \eqref{Carleman:coupled:phi:psi} can be bounded as follows:
	\begin{align}
		\label{observability:02}
		\begin{split}
			&\iint_{Q} e^{-2s\alpha}(\xi^4 |g^0|^2+\xi |g^1|^2)\,dx\,dt + \iint_{\Sigma} e^{-2s\alpha}(|g^0|^2+|g_\Gamma^1|^2)\,dS\,dt + \iint_{Q_{\omega}} e^{-2s\alpha} \xi^9 |\varphi|^2\,dx\,dt\\
			\leq & C\iint_{Q} e^{-2s\beta} (\gamma^4 |g^0|^2+\gamma |g^1|^2)\,dx\,dt + C\iint_{\Sigma} e^{-2s\beta} (|g_\Gamma^0|^2+|g_\Gamma^1|^2)\,dS\,dt\\
			 &+ C\iint_{Q_{\omega}} e^{-2s\beta }\gamma^9 |\varphi|^2\,dx\,dt.
		\end{split}
	\end{align}
	
	Secondly, in $[0,T/2]$ the have $\alpha=\beta$ and $\xi=\gamma$ and therefore we have 
	\begin{align}
		\label{observability:03}
		\begin{split}
			&\int_0^{T/2}\int_\Omega e^{-2s\beta} (\gamma^3 |\varphi|^2+\gamma |\nabla \varphi|^2 + \gamma^4 |\psi|^2 + \gamma^2 |\nabla \psi|^2 )\,dx\,dt\\
			&+\int_0^{T/2}\int_{\Gamma} e^{-2s\beta}\left( \gamma^3 |\varphi_\Gamma|^2 + \gamma |\nabla_\Gamma \varphi_\Gamma|^2 +\gamma^4 |\psi_\Gamma|^2 + \gamma^2 |\nabla_\Gamma \psi_\Gamma|^2 \right)\,dS\,dt\\
			=&\int_0^{T/2}\int_\Omega e^{-2s\alpha}(\xi^3 |\varphi|^2+\xi |\nabla \varphi|^2 + \xi^4 |\psi|^2 + \xi^2 |\nabla \psi|^2)\,dx\,dt \\
			+&\int_0^{T/2}\int_{\Gamma} e^{-2s\alpha}(\xi^3 |\varphi_\Gamma|^2 + \xi |\nabla \varphi_\Gamma|^2 + \xi^4 |\psi_\Gamma|^2 + \xi^2|\nabla_\Gamma \psi_\Gamma|^2 )\,dS\,dt. 
		\end{split}
	\end{align}
	
	Besides, in $[T/2,T]$ we use \eqref{observability:01} to deduce that 
	\begin{align}
		\nonumber 
		&\int_{T/2}^T \int_\Omega e^{-2s\beta}(\gamma^3 |\varphi|^2 + \gamma |\nabla \varphi|^2 +\gamma^4 |\psi|^2+ \gamma^2 |\nabla \psi|^2 )\,dx\,dt \\
		\nonumber 
		&+\int_{T/2}^{T} \int_\Gamma (|\varphi_\Gamma|^2 +|\nabla_\Gamma \varphi_\Gamma|^2 + |\psi_\Gamma|^2 +|\nabla_\Gamma \psi_\Gamma|^2 )\,dS\,dt\\
		\nonumber 
		\leq & C \|(\varphi,\varphi_\Gamma)\|_{L^2(T/2,T;\mathbb{H}^1)} ^2 + \|(\psi,\psi_\Gamma)\|_{L^2(T/2,T;\mathbb{H}^1)}^2\\ 
		\label{observability:04}
		\begin{split}
		\leq & C \int_{T/2}^T \int_\Omega e^{-2s\beta} (\gamma^4 |g^0|^2 + \gamma |g^1|^2)\,dx\,dt + C\int_{T/2}^T \int_\Gamma e^{-2s\beta} (|g_\Gamma^0|^2 + |g_\Gamma^0|^2)\,dS\,dt\\
		&+ C\int_{T/4}^{T/2} \int_\Omega e^{-2s\beta}(\gamma^3 |\varphi|^2 +\gamma |\nabla \varphi|^2)\,dx\,dt + C\int_{T/4}^{T/2}\int_\Gamma e^{-2s\beta} (\gamma^3 |\varphi_\Gamma|^2 + |\nabla_\Gamma \varphi_\Gamma|^2 )\,dS\,dt\\
		&+ C\int_{T/4}^{T/2} \int_\Omega e^{-2s\beta}(\gamma^4 |\psi|^2 +\gamma^2 |\nabla \psi|^2)\,dx\,dt + C\int_{T/4}^{T/2}\int_\Gamma e^{-2s\beta} (\gamma^4 |\psi_\Gamma|^2 + |\nabla_\Gamma \psi_\Gamma|^2 )\,dS\,dt.
		\end{split} 
	\end{align}
	
	Now, combining \eqref{observability:03},\eqref{observability:04}, \eqref{Carleman:coupled:phi:psi} together (with $s$ and $\lambda$ fixed) \eqref{Carleman:estimate:global:02} we obtain
	\begin{align*}
		\begin{split}
			&\iint_Q e^{-2s\beta} (\gamma^3 |\varphi|^2+\gamma |\nabla \varphi|^2 + \gamma^4 |\psi|^2 + \gamma^2 |\nabla \psi|^2)\,dx\,dt\\
			&+\iint_\Sigma e^{-2s\beta} (\gamma^3 |\varphi_\Gamma|^2 + \gamma |\nabla_\Gamma \varphi_\Gamma|^2 + \gamma^4 |\psi_\Gamma|^2 + \gamma^2 |\nabla \psi_\Gamma|^2 )\,dS\,dt\\
			\leq & C \iint_{Q} e^{-2s\beta} (\gamma^4 |g^0|^2 + \gamma|g^1|^2)\,dx\,dt + C\iint_\Sigma e^{-2s\beta}(|g_\Gamma^0|^2 + |g_\Gamma^1|^2)\,dS\,dt\\
			 &+ C\iint_{Q_\omega} e^{-2s\beta} \gamma^9 |\varphi|^2\,dx\,dt 
		\end{split}
	\end{align*}
	
	Finally, we take the minimum and maximum in the corresponding weight functions and we obtain \eqref{observability:varphi:psi}. This finishes the proof of Proposition \ref{proposition:observability}.
	
\end{proof}

\subsection{Null controllability}

Given the observability inequality , we restrict our attention to deduce a null controllability result of the linearized coupled system \eqref{coupled:linear:system:y}-\eqref{coupled:linear:system:z}.

Given $\sigma,\kappa\geq \frac{5}{2}$, we define $r_\star:=\min\{\sigma,\kappa\}$ and $r^\star:=\max\{\sigma,\kappa\}$. Then, let us define the functions
\begin{align*}
	\mu_1(t):=\exp \left(\frac{s\beta^*}{2(r_\star-2)} \right) (\gamma^*)^{-\frac{1}{r^\star-2}}\quad \mu_2(t):=\exp\left(\dfrac{3s\beta^*}{4(r_\star-1)} \right)(\gamma^*)^{-\frac{9}{8(r^\star-1)}}\quad \forall \,t\in (0,T).
\end{align*} 

We introduce the Banach space
\begin{align*}
	\mathcal{C}_{\sigma,\kappa}:=\{ (y,y_\Gamma ,z,z_\Gamma ,v)\,:\, & (e^{s\hat{\beta}}(\hat{\gamma})^{-2} y, e^{s\hat{\beta}} y_\Gamma)\in L^2(0,T;\mathbb{L}^2),\, (e^{s\hat{\beta}} (\hat{\gamma})^{-1/2} z, e^{s\hat{\beta}} z_\Gamma)\in L^2(0,T;\mathbb{L}^2), \\
	& \left. e^{s\beta^*}(\gamma^*)^{-3/2}(L(y)-\mathbbm{1}_\omega v, L_\Gamma (y,y_\Gamma))\in L^2(0,T;\mathbb{L}^2), \right.\\
	& \left. e^{s\beta^*}(\gamma^*)^{-2}(L^*(z)-\mathbbm{1}_{\mathcal{O}}y,L_\Gamma^*(z,z_\Gamma)-\nabla_\Gamma (\mathbbm{1}_{\mathcal{G}} \nabla_\Gamma y_\Gamma))\in L^2(0,T;\mathbb{L}^2), \right.\\
	& \left. e^{s\hat{\beta}}\hat{\gamma}^{-9/2} v\in L^2(Q_\omega),\, (z(\cdot,T),z_\Gamma(\cdot,T))=(0,0)\text{ in }\Omega\times \Gamma, \right.\\
	&\mu_i\cdot (y,y_\Gamma,z,z_\Gamma)\in \mathbb{D}^2,\,i=1,2\}
	\end{align*}
endowed by its natural norm.

\begin{remark}
Let us point out that if $(y,y_\Gamma,z,z_\Gamma) \in \mathcal{C}_{\sigma,\kappa}$, then we necesarily have  $$(z(\cdot,0),z_\Gamma(\cdot,0))=(0,0)\text{ in }\Omega\times \Gamma.$$

This is due that the weight $e^{s\hat{\beta}}$ blows up at $t\to 0^+$. 
\end{remark}

We have the following null controllability result for the linearized system:

\begin{proposition}
	\label{proposition:null:control:C} Let $\sigma,\kappa> \frac{5}{2}$.
	Suppose that 
	\begin{align}
		\label{conditions:sources:f}
		e^{s\beta^*}(\gamma^*)^{-3/2}(f^0,f_\Gamma^0)\in L^2(0,T;\mathbb{L}^2),\quad e^{s\beta^*}(\gamma^*)^{-2}(f^1,f_\Gamma^1)\in L^2(0,T;\mathbb{L}^2).
	\end{align}
	
	There exists $v\in L^2(Q_\omega)$ such that if $(y,y_\Gamma,z,z_\Gamma)$ is the solution of \eqref{coupled:linear:system:y}-\eqref{coupled:linear:system:z} with control $v$ and source terms $(f^0,f_\Gamma^0)$ and $(f^1,f_\Gamma^1)$, then $(y,y_\Gamma,z,z_\Gamma)$ belongs to $\mathcal{C}_{\sigma,\kappa}$.
\end{proposition} 
\begin{proof}
	Let us define the space
	\begin{align*}
		P_0=\{(\varphi,\varphi_\Gamma,\psi,\psi_\Gamma)\,:&\, (\varphi,\varphi_\Gamma),(\psi,\psi_\Gamma)\in L^2(0,T;\mathbb{H}^1)\,,\, (L^*(\varphi)-\psi \mathbbm{1}_{\mathcal{O}}) \in L^2(Q), \\
		 & L_\Gamma^*(\varphi,\varphi_\Gamma)-\nabla_\Gamma (\mathbbm{1}_{\mathcal{G}} \nabla_\Gamma \psi_\Gamma) \in L^2(\Sigma) \, , \, L(\psi)\in L^2(Q), \\
		 & L_\Gamma(\psi,\psi_\Gamma)\in L^2(\Sigma)\}
	\end{align*}
	
	Consider the bilinear form $a:P_0\times P_0 \to \mathbb{R}$ given by 
	\begin{align*}
		&a((\hat{\varphi},\hat{\varphi}_\Gamma,\hat{\psi},\hat{\psi}_\Gamma),(\varphi,\varphi_\Gamma,\psi,\psi_\Gamma))\\
		=& \iint_{Q} e^{-2s\hat{\beta}} \hat{\gamma}^4 (L^*(\hat{\varphi})- \mathbbm{1}_{\mathcal{O}} \hat{\psi}) (L^*(\varphi)-\mathbbm{1}_\mathcal{O} \psi)\,dx\,dt+\iint_{Q} e^{-2s\hat{\beta}}\hat{\gamma} L(\hat{\psi}) L(\psi)\,dx\,dt\\
		&+\iint_\Sigma e^{-2s\hat{\beta}} (L_\Gamma^*(\hat{\varphi},\hat{\varphi}_\Gamma) -\nabla_\Gamma (\mathbbm{1}_\mathcal{G} \nabla_\Gamma \hat{\psi}_\Gamma))(L_\Gamma^*(\varphi,\varphi_\Gamma) -\nabla_\Gamma (\mathbbm{1}_\mathcal{G} \nabla_\Gamma \psi_\Gamma))\,dS\,dt\\
		&+\iint_\Sigma e^{-2s\hat{\beta}} L_\Gamma (\hat{\psi},\hat{\psi}_\Gamma)L_\Gamma (\psi,\psi_\Gamma)\,dS\,dt + \iint_{Q_\omega} e^{-2s\hat{\beta}} \hat{\gamma}^9 \hat{\varphi}\varphi \,dx\,dt.
	\end{align*}
	
	Moreover, define the linear form $G:P_0\to \mathbb{R}$ by 
	\begin{align*}
		G(\varphi,\varphi_\Gamma,\psi,\psi_\Gamma):=\iint_{Q} (f^0\varphi + f^1\psi)\,dx\,dt + \iint_{\Sigma} (f_\Gamma^0 \varphi_\Gamma +f_\Gamma^1 \psi_\Gamma)\,dSdt.
	\end{align*}
	
	Thanks to observability inequality \eqref{observability:varphi:psi}, it is easy to see that $$\|(\varphi,\varphi_\Gamma,\psi,\psi_\Gamma)\|_P:= [a((\varphi,\varphi_\Gamma,\psi,\psi_\Gamma),(\varphi,\varphi_\Gamma,\psi,\psi_\Gamma))]^{1/2}$$
	defines a norm on $P_0$.
	
	Let $P$ be the completion of $P_0$ by the norm $\|\cdot\|_P$. Besides, it is clear that $a(\cdot,\cdot)$ is continuous and coercive on $P\times P$. We also notice that for each $(\varphi,\varphi_\Gamma,\psi,\psi_\Gamma)\in P_0$, we have 
	\begin{align*}
		G(\varphi,\varphi_\Gamma,\psi,\psi_\Gamma)\leq M \|(\varphi,\varphi_\Gamma,\psi,\psi_\Gamma)\|_P,
	\end{align*}
	where $M$ is given by 
	\begin{align*}
		\begin{split} 
		M:=&C_{\text{obs}}^{1/2}\left( \|e^{s\beta^*}(\gamma^*)^{-3/2}(f^0,f_\Gamma^0)\|_{L^2((0,T);\mathbb{L}^2)} + \|e^{s\beta^*}(\gamma^*)^{-2} (f^1,f_\Gamma^1)\|_{L^2()(0,T);\mathbb{L}^2)}  \right), 
		\end{split}
	\end{align*}
	where $C_{\text{obs}}>0$ is the constant used in the Observability inequality \eqref{observability:varphi:psi}. Then, by Hahn-Banach Theorem, $G$ can be extended to a continuous functional on $P$. Consequently, by Lax-Milgram Lemma, the problem: Find $(\hat{\varphi},\hat{\varphi}_\Gamma,\hat{\psi},\hat{\psi}_\Gamma)\in P$ such that 
	\begin{align}
		\label{null:control:01}
		a((\hat{\varphi},\hat{\varphi}_\Gamma,\hat{\psi},\hat{\psi}_\Gamma),(\varphi,\varphi_\Gamma,\psi,\psi_\Gamma))=G(\varphi,\varphi_\Gamma,\psi,\psi_\Gamma)\quad \forall (\varphi,\varphi_\Gamma,\psi,\psi_\Gamma)\in P
	\end{align}
	has a unique solution. Now, define 
	\begin{align*}
		&\hat{y}:=e^{-2s\hat{\beta}}\hat{\gamma}^4(L^*(\hat{\varphi})-\hat{\psi}\mathbbm{1}_{\mathcal{O}}),\quad \hat{z}:=e^{-2s\hat{\beta}}\hat{\gamma} L(\psi),\\
		&\hat{y}_\Gamma:=e^{-2s\hat{\beta}}(L_\Gamma^*(\hat{\varphi},\hat{\varphi}_\Gamma) - \nabla_\Gamma (\mathbbm{1}_{\mathcal{G}} \nabla_\Gamma \hat{\psi}_\Gamma)),\quad \hat{z}_\Gamma:=e^{-2s\hat{\beta}} L_\Gamma (\hat{\psi},\hat{\psi}_\Gamma),\\
		& \hat{v}:= -e^{-2s\hat{\beta}} \hat{\gamma}^9 \hat{\varphi}.
	\end{align*}
	
	Then, the identity \eqref{null:control:01} reads as follows:
	\begin{align*}
		\begin{split}
		&\iint_{Q} (\hat{y}(L^*(\varphi)-\psi\mathbbm{1}_{\mathcal{O}}) +\hat{z} L(\psi))\,dx\,dt\\
		&+\iint_{\Sigma} (\hat{y}_\Gamma (L_\Gamma^*(\varphi,\varphi_\Gamma)-\nabla_\Gamma (\mathbbm{1}_{\mathcal{G}} \nabla_\Gamma \psi_\Gamma)) +\hat{z}_\Gamma L_\Gamma (\psi,\psi_\Gamma) )\,dS\,dt\\
		=&\iint_{Q_\omega}\hat{v}\varphi \,dx\,dt + \iint_{Q} (f^0\varphi +f^1\psi)\,dx\,dt + \iint_{\Sigma }(f_\Gamma^0 \varphi_\Gamma +f_\Gamma^1 \psi_\Gamma)\,dS\,dt,
		\end{split}
	\end{align*}
	which implies that $(\hat{y},\hat{y}_\Gamma,\hat{z},\hat{z}_\Gamma)$ is a distributional solution of \eqref{coupled:linear:system:y}-\eqref{coupled:linear:system:z} where $\hat{v}$ is a control driving the state $(\hat{y},\hat{y}_\Gamma,\hat{z},\hat{z}_\Gamma)$ to the rest at $t=T$. Indeed, by condition \eqref{conditions:sources:f}, we deduce that $(\hat{y},\hat{y}_\Gamma,\hat{z},\hat{z}_\Gamma)\in \mathbb{E}$ is a strong solution. On the other hand, we notice that
	\begin{align}
		\label{regu:sol:01}
		\begin{split} 
		&\iint_{Q} e^{2s\hat{\beta}}((\hat{\gamma})^{-4} |\hat{y}|^2 + (\hat{y})^{-2}|\hat{z}|^2)\,dx\,dt + \iint_{\Sigma} e^{2s\hat{\beta}}(|\hat{y}_\Gamma|^2+|\hat{z}_\Gamma|^2)\,dS\,dt \\
		&+ \iint_{Q_\omega} e^{2s\hat{\beta}}\hat{\gamma}^{9} |\hat{v}|^2\,dx\,dt
		=a((\hat{\varphi},\hat{\varphi}_\Gamma,\psi,\psi_\Gamma),(\hat{\varphi},\hat{\varphi}_\Gamma,\hat{\psi},\hat{\psi}_\Gamma))<+\infty.
		\end{split} 
	\end{align}
	
	It remains to prove that 
	$
		\mu_i\cdot  (\hat{y},\hat{y}_\Gamma,\hat{z},\hat{z}_\Gamma)\in \mathbb{D}^2
	$, for $i=1,2$. To do this, let us define the new variables 
	\begin{align*}
		\begin{split}
		(y^*,y_\Gamma^*,z^*,z_\Gamma^*):=&\mu_1\cdot (\hat{y},\hat{y}_\Gamma,\hat{z},\hat{z}_\Gamma),\quad
		(f_*^0,f_{\Gamma,*}^0)=\mu_1\cdot (f^0+\mathbbm{1}_\omega \hat{v},f_\Gamma^0),\\ &(f_*^1,f_{*,\Gamma}^1)=\mu_1\cdot (f^1,f_\Gamma^1).
		\end{split}
	\end{align*}
	
	Then, it is easy to see that $(y^*,y_\Gamma^*,z^*,z_\Gamma^*)$ is (at least) a distributional solution of the coupled system 
	\begin{align}
		\label{regu:sol:02}
		\begin{cases}
			L(y^*)=f_*^0 +\mu_{1,t} \hat{y}&\text{ in }Q,\\
			L_\Gamma(y^*,y_\Gamma^*)=f_{\Gamma,*}^0 + \mu_{1,t} \hat{y}_\Gamma&\text{ on }\Sigma,\\
			y^*\big|_\Gamma=y_\Gamma^*&\text{ on }\Sigma,\\
			(y^*(\cdot,0),y_\Gamma^*(\cdot,0))=(0,0)&\text{ in }\Omega\times \Gamma.
		\end{cases}
	\end{align}
	
	\begin{align}
		\label{regu:sol:03}
			\begin{cases}
			L^*(z^*)=f_* ^1 + \mathbbm{1}_{\mathcal{O}}y^* + \mu_{1,t} \hat{z}&\text{ in }Q,\\
			L_{\Gamma}^{*}(z^*,z_\Gamma^*)=f_{\Gamma,*}^1 + \nabla_\Gamma \cdot (\mathbbm{1}_{\mathcal{G}} \nabla_\Gamma \hat{y}_\Gamma) + \mu_{1,t} \hat{z}_\Gamma &\text{ on }\Sigma,\\
			z^*\big|_\Gamma =z_\Gamma^*&\text{ on }\Sigma,\\
			(z^*(\cdot,T),z_\Gamma^*(\cdot,T))=(0,0)&\text{ in }\Omega\times \Gamma.
		\end{cases}
	\end{align}
	
	Now, notice that for each $\sigma,\kappa\geq \frac{5}{2}$, we obtain  $\frac{1}{2(r_*-2)}<1$. Then, straightforward computations show that there exists a constant $C>0$ such that 
	\begin{align}
		\label{regu:sol:04}
		\mu_{1,t}\leq Ce^{s\hat{\beta}}\text{ in }(0,T).
	\end{align}
	
	From \eqref{regu:sol:01} and \eqref{regu:sol:04}, we deduce that 
	\begin{align*}
		(f_*^0 + \mu_{1,t} \hat{y}, f_{\Gamma,*}^0 + \mu_{1,t} \hat{y}_\Gamma)\in L^2(0,T;\mathbb{L}^2),\quad 
		(f_* ^1 + \mu_{1,t} \hat{z}, 
		f_{\Gamma,*}^1 + \mu_{1,t} \hat{z}_\Gamma) \in L^2(0,T;\mathbb{L}^2).
	\end{align*}
	
	This means that $(y^*,y_\Gamma^*,z^*,z_\Gamma^*)$ is indeed a strong solution of \eqref{regu:sol:02}-\eqref{regu:sol:03} with the regularity property $		(y^*,y_\Gamma^*,z^*,z_\Gamma^*)\in \mathbb{D}^2
	$. Similarly, we can prove that $\mu_2 (\hat{y},\hat{y}_\Gamma,\hat{z},\hat{z}_\Gamma)\in \mathbb{D}^2$. Therefore, $(\hat{y},\hat{y}_\Gamma,\hat{z},\hat{z}_\Gamma)\in \mathcal{C}_{\sigma,\kappa}$, which completes the proof of Proposition \ref{proposition:null:control:C}.
\end{proof}

\section{Proof of the main result}
\label{section:proof:main:theorem}
In this section, we prove the local null controllability of the coupled system \eqref{Coupled:system:nonlinear:y}-\eqref{Coupled:system:nonlinear:z}. In order to do this, we recall a local inversion mapping theorem in Banach spaces (see \cite{Alekseev1987Optimal}, page 107). 

\begin{theorem}
	\label{thm:Alexseev}
	Let $\mathcal{B}_1$ and $\mathcal{B}_2$ be two Banach spaces and let $\mathcal{A}:\mathcal{B}_1\to \mathcal{B}_2$ be a $C^1(\mathcal{B}_1; \mathcal{B}_2)$ function. 
	Suppose that $b_1\in \mathcal{B}_1$, $b_2=\mathcal{A}(b_1)$ and that 
	$\mathcal{A}'(b_1):\mathcal{B}_1\to \mathcal{B}_2$ is surjective. Then, there exists
	$\delta>0$ such that, for every $b'\in \mathcal{B}_2$ satisfying $\|b'-b_2\|_{\mathcal{B}_2 }<\delta$, there exists a solution of the equation 
	\begin{align*}
		\mathcal{A}(b)=b',\quad b\in \mathcal{B}_1.
	\end{align*}
	
	Moreover, there exists a constant $C>0$ such that 
	\begin{align*}
		\|b_1-b\|_{\mathcal{B}_1}\leq C\|b_2-b'\|_{\mathcal{B}_2}.
	\end{align*}
	
\end{theorem}
	
Now, let us prove the main theorem of this paper:
	
\begin{proof}[Proof of Theorem \ref{thm:nonlinear:controllability:coupled:system}]
	We just analyze the case $d=3$. The case $d=2$ is left to the reader. For $p,q$ given in {\bf (H1)}, we define the Banach spaces
	\begin{align*}
		\mathcal{B}_1:=\mathcal{C}_{p,q},\quad \mathcal{B}_2:= L^2(e^{2s\beta^*}(\gamma^*)^{-3}(0,T);\mathbb{L}^2)\times L^2(e^{2s\beta^*}(\gamma^*)^{-4}(0,T);\mathbb{L}^2)\times \mathbb{L}^2\times \mathbb{L}^2.
	\end{align*}
	
	Consider the operator $\mathcal{A}:\mathcal{B}_1\to \mathcal{B}_2$ defined by 
	\begin{align*}
		\mathcal{A}(y,y_\Gamma,z,z_\Gamma,v):=&(L(y)+|y|^{p-1}y-f-\mathbbm{1}_\omega v,L_\Gamma (y,y_\Gamma)+|y_\Gamma|^{q-1}y_\Gamma-f_\Gamma,\\
		&L^*(z)+p|y|^{p-2}yz-\mathbbm{1}_{\mathcal{O}}y, L_\Gamma^*(z,z_\Gamma)+q|y_\Gamma|^{q-2}y_\Gamma z_\Gamma -\nabla_\Gamma \cdot (\mathbbm{1}_{\mathcal{G}} \nabla_\Gamma y_\Gamma),\\
		&y(\cdot,0),y_\Gamma(\cdot,0),z(\cdot,T),z_\Gamma(\cdot,T)).
	\end{align*}	
	
	In addition, we choose $b_1=(0,0,0,0,0)$ and $b_2=(0,0,0,0,0,0)$. Then, to apply the Proposition \ref{proposition:null:control:C}, it is sufficient to check the following two assertions:
	\begin{itemize}
		\item[(a)] The mapping $\mathcal{A}'(0,0,0,0,0):\mathcal{B}_1\to \mathcal{B}_2$ is surjective,
		\item[(b)] $\mathcal{A}$ is an operator of class $C^1$ from $\mathcal{B}_1$ to $\mathcal{B}_2$.  
	\end{itemize} 
	
	For each $(y,y_\Gamma,z,z_\Gamma,v),(\hat{y},\hat{y}_\Gamma,\hat{z},\hat{z}_\Gamma,\hat{v})\in \mathcal{B}_1$, we see that 
	\begin{align*}
	\begin{split}
		&[\mathcal{A}'(y,y_\Gamma,z,z_\Gamma,v)](\hat{y},\hat{y}_\Gamma,\hat{z},\hat{z}_\Gamma,\hat{v})\\
		=&(L(\hat{y})+py^{p-1}\hat{y}-\mathbbm{1}_\omega \hat{v},L_\Gamma (\hat{y},\hat{y}_\Gamma)+qy_\Gamma^{q-1}\hat{y}_\Gamma,L^*(\hat{z})+p(p-1)y^{p-2}\hat{y}z+py^{p-1}\hat{z}-\mathbbm{1}_{\mathcal{O}} \hat{y},\\
		 &L_\Gamma^*(\hat{z},\hat{z}_\Gamma)+q(q-1)y_\Gamma^{q-2}\hat{y}_\Gamma z_\Gamma + qy_\Gamma^{q-1}\hat{z}_\Gamma -\nabla_\Gamma \cdot (\mathbbm{1}_{\mathcal{G}}\nabla_\Gamma \hat{y}_\Gamma),\hat{y}(\cdot,0),\hat{y}_\Gamma(\cdot,0),\hat{z}(\cdot,T),\hat{z}_\Gamma(\cdot,T))
	\end{split}
	\end{align*}
	
	In particular, for $(y,y_\Gamma,z,z_\Gamma,v)=(0,0,0,0,0)$, we have
	\begin{align*}
		&\mathcal{A}'(0,0,0,0,0)(\hat{y},\hat{y}_\Gamma,\hat{z},\hat{z}_\Gamma,\hat{v})\\
		=& (L(\hat{y})-\mathbbm{1}_\omega \hat{v},L_\Gamma(\hat{y},\hat{y}_\Gamma),L^*(\hat{z})-\mathbbm{1}_\mathcal{O} \hat{y},L_\Gamma^*(\hat{z},\hat{z}_\Gamma)-\nabla_\Gamma (\mathbbm{1}_{\mathcal{G}} \nabla_\Gamma \hat{y}_\Gamma),\\
		&\hat{y}(\cdot,0),\hat{y}_\Gamma (\cdot,0),\hat{z}(\cdot,T),\hat{z}_\Gamma(\cdot,T)).
	\end{align*} 
	
	Then, assertion (a) follows from the null controllability property of the linear system in Proposition \ref{proposition:null:control:C}.
	
	To check (b), we need to prove the following assertions: 
	\begin{itemize}
		\item[$(i)$] $e^{s\beta^*}(\gamma^*)^{-3/2}(y^{p-1}\tilde{y},y_\Gamma^{q-1}\tilde{y}_\Gamma)\in L^2(0,T;\mathbb{L}^2)$,
		\item[$(ii)$] $e^{s\beta^*}(\gamma^*)^{-2}(y^{p-2}\tilde{y}z+y^{p-1}\tilde{z},y_\Gamma^{q-2}\tilde{y}_\Gamma,z_\Gamma+y_\Gamma^{q-1}\tilde{z}_\Gamma)\in L^2(0,T;\mathbb{L}^2)$. 
	\end{itemize}
	
	We introduce the temporary notation $E(X):=C^0([0,T];H^1(X)\cap L^2(0,T;H^2(X))$, where $X=\Omega$ or $X=\Gamma$. Let us prove the assertion $(i)$. Firstly, by H\"older's inequality we have  
	\begin{align*}
		&\|e^{s\beta^*}(\gamma^*)^{-3/2} y^{p-1}\tilde{y}\|_{L^2(Q)}\\
		\leq & \|e^{\frac{3}{4}s\beta^*} (\gamma^*)^{-9/8}y^{p-1}\|_{L^{8/3}(Q)}\cdot \|e^{\frac{s\beta^*}{4}} (\gamma^*)^{-3/8} \tilde{y}\|_{L^8(Q)}\\
		\leq & \left\|\exp\left(\frac{3s\beta^*}{4(p-1)} \right)(\gamma^*)^{-\frac{8}{9(p-1)}}y\right\|_{L^{\frac{8}{3}(p-1)}(Q)}^{p-1}\cdot \|e^{\frac{s\beta^*}{4} (\gamma^*)^{-3/8}} \tilde{y}\|_{L^8(Q)}
	\end{align*}
	From the continuous embeddings $\mathbb{D}\hookrightarrow L^{8}(0,T;\mathbb{L}^8)$, we deduce that 
	\begin{align*}
		&\|e^{s\beta^*}(\gamma^*)^{-3/2} (y^{p-1}\tilde{y},y_\Gamma^{q-1}\tilde{y}_\Gamma)\|_{L^2(0,T;\mathbb{L}^2)}\\
		\leq & C \left( \left\| \exp\left(\dfrac{3s\beta^*}{4(p-1)}\right)(\gamma^*)^{-\frac{9}{8(p-1)}} y \right\|_{E(\Omega)}^{p-1} 
		+\left\|\exp \left(\dfrac{3s\beta^*}{4(q-1)}\right)(\gamma^*)^{-\frac{9}{8(q-1)}}y_\Gamma  \right\|_{E(\Gamma)}^{q-1}\right)\\
		&\cdot \|(\tilde{y},\tilde{y}_\Gamma)\|_{\mathcal{C}_{p,q}},
	\end{align*}
	which proves $(i)$. In the same manner, it is easy to see that 
	\begin{align}
		\label{ineq:ii:01}
		\begin{split} 
		&\left\| e^{s\beta^*}(\gamma^*)^{-2}(y^{p-1}\tilde{z},y_\Gamma^{q-1}\tilde{z}_\Gamma) \right\|_{L^2(0,T;\mathbb{L}^2)}\\
		\leq & C\left( \left\|\exp\left(\frac{3s\beta^*}{4(p-1)} \right) (\gamma^*)^{-\frac{3}{2(p-1)}} y \right\|_{E(\Omega)}^{p-1} + \left\| \exp\left(\frac{3s\beta^*}{4(q-1)} \right) (\gamma^*)^{-\frac{3}{2(q-1)}}y_\Gamma \right\|_{E(\Gamma)}^{q-1} \right)\\
		&\cdot \|(\tilde{z},\tilde{z}_\Gamma)\|_{\mathcal{C}_{p,q}} 
		\end{split}  
	\end{align}
	
	On the other hand, we notice that, by H\"older's inequality we have 
	\begin{align*}
		&\|e^{s\beta^*}(\gamma^*)^{-2}y^{p-2}\tilde{y}z\|_{L^2(Q)}\\
		\leq & \|e^{\frac{1}{2}s\beta^*}(\gamma^*)^{-1}y^{p-2}\|_{L^4(Q)}\cdot \|e^{\frac{1}{4}}(\gamma^*)^{-1/2}\tilde{y} \|_{L^8(Q)}\cdot \|e^{\frac{1}{4}s\beta^*}(\gamma^*)^{-1/2}z\|_{L^8(Q)}\\
		\leq & C \left\|\exp\left(\dfrac{s\beta^*}{2(p-2)} \right)(\gamma^*)^{-\frac{1}{p-2}}y  \right\|_{E(\Omega)}^{p-2}\|e^{\frac{s\beta^*}{4}(\gamma^*)^{-1/2}}\tilde{y}\|_{E(\Omega)}\cdot \|e^{\frac{s\beta^*}{4}}(\gamma^*)^{-1/2}z\|_{E(\Omega)}
	\end{align*}
	
	Then, we conclude that
	\begin{align}
		\label{ineq:ii:02}
		\begin{split}
			&\|e^{s\beta^*}(\gamma^*)^{-2}(y^{p-2}\tilde{y}z,y_\Gamma^{q-2}\tilde{y}_\Gamma z_\Gamma) \|_{L^2(0,T;\mathbb{L}^2)}\\
			\leq &C \left(\left\| \exp \left(\dfrac{s\beta^*}{2(p-2)} \right) (\gamma^*)^{-\frac{1}{p-2}}y \right\|_{E(\Omega)}^{p-2} + \left\| \exp\left(\dfrac{s\beta^*}{2(q-2)}\right)(\gamma^*)^{-\frac{1}{q-2}} y_\Gamma  \right\|_{E(\Gamma)}^{q-2} \right)\\
			&\cdot \|(\tilde{y},\tilde{y}_\Gamma)\|_{\mathcal{C}_{p,q}}\cdot \|(z,z_\Gamma)\|_{\mathcal{C}_{p,q}}.
		\end{split}
	\end{align}
	
	From \eqref{ineq:ii:01} and \eqref{ineq:ii:02}, we proved the assertion $(ii)$. 

This proves that $\mathcal{A}\in C^1(\mathcal{B}_1,\mathcal{B}_2)$. Finally, the Theorem \ref{thm:Alexseev} guarantees the existence of $\delta>0$ such that
	\begin{align*}
		\|e^{s\beta^*}(\gamma^*)^{-3/2}(f,f_\Gamma)\|_{L^2(0,T;\mathbb{L}^2)}\leq \delta,
	\end{align*}
	one can find a couple $(y,y_\Gamma,z,z_\Gamma,v)\in \mathcal{C}_{p,q}$. In particular, this means that the control $v$ drives the associated trajectory $(y,y_\Gamma,z,z_\Gamma)$ to the state
	\begin{align*}
		z(\cdot,0)=0\text{ in }\Omega\text{ and }z_\Gamma(\cdot,0)=0\text{ on }\Gamma.
	\end{align*}  
	
	This ends the proof of the Theorem \ref{thm:nonlinear:controllability:coupled:system}.
\end{proof}

\section{Further comments}
\label{section:further:comments}
In this paper, we have considered an insensitizing control problem that involves tangential gradient terms for the heat equation with dynamic boundary conditions. Nevertheless, there are many interesting and challenging problems related to the insensitizing control problems for such systems. Some of them are mentioned below.
\subsection{On the initial data}

In our main results, we have chosen a particular initial data on \eqref{eq:intro:01} to insensitize the functional $\Phi$, namely $(y^0,y_\Gamma^0)=(0,0)$. As we said before, a study of possible initial conditions for which $\Phi$ is insensitive can also be considered following the steps of \cite{deteresa2009Identification}. However, with a slight modification on the functional $\Phi$, it is possible to establish an insensitizing control result for any initial data in $\mathbb{L}^2$. More precisely, for fixed $t_0>0$ and $\mathcal{O}\subseteq \Omega$ and $\mathcal{G}\subseteq \Gamma$ being two nonempty subsets, we define the functional
\begin{align*}
	\Phi_1(y,y_\Gamma)=\dfrac{1}{2}\int_{t_0}^T\int_{\mathcal{O}} |y(x,t;\tau_1,\tau_2,v)|^2\,dx\,dt + \dfrac{1}{2}\int_{t_0}^T\int_{\mathcal{G}} |\nabla y_\Gamma(x,t;\tau_1,\tau_2,v)|^2\,dS\,dt. 
\end{align*} 

In this case, the notion of insensitizing controls for $\Phi_1$ is analogous to the Definition \ref{def:Phi:insen}. Then we can establish an analogous result as in Theorem \ref{Thm:main:result}.

\subsection{A sentinel involving the gradient and the tangential gradient}

We can consider a situation where the sentinel depends on the gradient in the bulk and the tangential gradient on the boundary. More precisely, given $\mathcal{O}\subset \Omega$ and $\mathcal{G}\subset \Gamma$ be two nonempty subsets, we consider the quadratic functional $\Psi$ defined by 
\begin{align*}
	\Psi(y,y_\Gamma):=\dfrac{1}{2}\iint_{Q_{\mathcal{O}}} |\nabla y(x,t;\tau_1,\tau_2,u)|^2\,dx\,dt + \dfrac{1}{2}\iint_{\Sigma_{\mathcal{G}}}|\nabla_\Gamma y_\Gamma(x,t;\tau_1,\tau_2,u)|^2\,dS\,dt,
\end{align*}
with $(y,y_\Gamma)$ being the corresponding solution of \eqref{eq:intro:01}. In this case, the insensitivity notion is the same as in Definition \ref{def:Phi:insen} applied to $\Psi$. Besides, this notion is equivalently to prove a controllability result for the coupled system 
\begin{align}
	\label{Coupled:insen:02:y}
	\begin{cases}
		L(y)+|y|^{p-1}y=f+\mathbbm{1}_\omega v&\text{ in }Q,\\
		L_\Gamma(y,y_\Gamma)+|y_\Gamma|^{q-1}y_\Gamma=f_\Gamma&\text{ on }\Sigma,\\
		y\big|_\Gamma=y_\Gamma &\text{ on }\Sigma,\\
		(y(\cdot,0),y_\Gamma(\cdot,0))=(y^0,y_\Gamma^0)&\text{ in }\Omega\times \Gamma.
	\end{cases}
\end{align}
\begin{align}
	\label{Coupled:insen:02:z}
	\begin{cases}
		L^*(z)+p|y|^{p-2}yz=\nabla\cdot (\mathbbm{1}_{\mathcal{O}}\nabla y)&\text{ in }Q,\\
		L_\Gamma^*(z,z_\Gamma) + q|y_\Gamma|^{q-2}y_\Gamma z_\Gamma=\nabla_\Gamma \cdot (\mathbbm{1}_{\mathcal{G}} \nabla_\Gamma y_\Gamma )&\text{ on }\Sigma,\\
		z\big|_\Gamma =z_\Gamma&\text{ on }\Sigma,\\
		(z(\cdot,T),z_\Gamma(\cdot,T))=(0,0)&\text{ in }\Omega\times \Gamma.
	\end{cases}
\end{align}

More precisely, we are interested in finding conditions to guarantee the existence of a control $v\in L^2(Q_\omega)$ such that the associated coupled system \eqref{Coupled:insen:02:y}-\eqref{Coupled:insen:02:z} satisfies
\begin{align*}
	z(\cdot,0)=0\text{ in }\Omega\text{ and }z_\Gamma(\cdot,0)=0\text{ on }\Gamma.
\end{align*}

Then, following the steps in \cite{guerrero2007null}, we are interested in deriving a Carleman estimate for the equations satisfied by $(\zeta,\zeta_\Gamma)=(\Delta \varphi,\Delta_\Gamma \varphi)$. However, the relation of these terms on the boundary is given by the following identity: 
\begin{align*}
	\Delta \varphi=\Delta_\Gamma \varphi_\Gamma+H\pnu \varphi +\partial_{\nu \nu} \varphi \quad \forall\, (\varphi,\varphi_\Gamma)\in H^3(\Omega)\times H^2(\Gamma)\text{ with }\varphi\big|_\Gamma=\varphi_\Gamma, 
\end{align*}
where $H$ is the mean curvature of $\Gamma$ and $\partial_{\nu\nu} \varphi$ is the second normal derivative, see for instance \cite[Proposition 5.4.12]{henrot2018Shape}. This means that we need to perform Carleman estimates where the side condition $\zeta\big|_\Gamma=\zeta_\Gamma$ is not satisfied. This seems to be a challenging problem and it will be considered in a forthcoming paper.


\section*{Acknowledgments}

Maur\'icio Santos has been funded by CNPq, CAPES and MathAmSud SCIPinPDEs. Nicol\'as Carre\~{n}o has been funded by ANID FONDECYT 1211292. Roberto Morales has been funded under the Grant QUALIFICA by Junta de Andaluc\'ia grant number QUAL21 005 USE. 


  \bibliography{biblio.bib}
  \bibliographystyle{plain}

\end{document}